\documentclass[pdflatex,sn-mathphys]{sn-jnl}
\AtBeginDocument{}

\jyear{2023}%

\theoremstyle{thmstyleone}%
\newtheorem{theorem}{Theorem}
\newtheorem{proposition}[theorem]{Proposition}%
\newtheorem{lemma}[theorem]{Lemma}%

\theoremstyle{thmstyletwo}%
\newtheorem{example}{Example}%
\newtheorem{remark}{Remark}%

\theoremstyle{thmstylethree}%
\newtheorem{definition}{Definition}%
\newtheorem{assumptions}{Assumptions}%

\usepackage{mathrsfs}
\usepackage[shortlabels]{enumitem}
\usepackage{dsfont}

\newcommand{\pars}[1]{\left(  #1 \right)}
\newcommand{\abs}[1]{\left\vert   #1 \right\vert }
\newcommand{\setof}[1]{\left\{ #1\right\}}
\newcommand{\sqkets}[1]{\left[ #1 \right]}

\newcommand{\esper}[2]{\mathbb{E}_{#1} \left[  #2 \right]}
\newcommand{\espero}[1]{\mathbb{E} \left[  #1 \right] }
\newcommand{\norm}[1]{ {\left| \left| #1 \right| \right| } } 
\newcommand{\mnorm}[1]{ {\left| \left|\left|   #1 \right| \right| \right| } } 

\newcommand{\rr}{\mathbb{R}}

\newcommand{\cx}{\mathcal{X}}
\newcommand{\bx}{\mathbf{x}}
\newcommand{\ap}[2]{\left\langle #1, #2 \right\rangle}
\newcommand{\Psitx}[2]{\exp \pars{ - \int_0^{#1} \beta \pars{\varphi^{s} \pars{#2} } ds }}

\usepackage{ulem}

\usepackage{pgfplots}
\pgfplotsset{compat=newest}
\usepackage{tikz}
\usetikzlibrary{calc}
\usetikzlibrary{arrows.meta}

\begin{document}

\title[ ]{Exponential ergodicity of a degenerate age-size piecewise deterministic process}


\author*[1]{\fnm{Ignacio} \sur{Madrid}}\email{ignacio.madrid-canales@polytechnique.edu}

\affil*[1]{\orgdiv{CMAP}, \orgname{\'Ecole polytechnique, Institut Polytechnique de Paris}, \orgaddress{\street{Route de Saclay}, \city{Palaiseau}, \postcode{91120}, \country{France}}}


\abstract{
{
	We study the long-time behaviour of a non conservative piecewise deterministic measure-valued Markov process modelling the proliferation of an age-and-size structured population, which generalises the ``adder" model of bacterial growth. Firstly, we prove the existence of eigenelements of the associated infinitesimal generator, which are used to bring ourselves back to the study of a conservative Markov process using a Doob $h$-transform. Finally, we obtain the exponential ergodicity of the process via drift-minorisation arguments. Specifically, we show the ``petiteness" of the compact sets of the state space. This permits to circumvent the difficulties encountered when trying to construct mixing trajectories at a fixed uniform time on an unbounded two-dimensional space with only advection and degenerate jump terms.}
}

\keywords{Exponential ergodicity, Harris' Ergodic Theorem, Minorisation condition, Non-conservative semigroups, Age-size structured equation, Degenerate PDMP}

\pacs[MSC Classification]{60J25, 45C05, 45K05, 92D25}

\maketitle

\section{Introduction}

The need to include age as a structuring variable in the description of population dynamics has come to be a useful strategy for modellers searching to account for non-Markovian behaviours in a Markovian setting. In particular, in the context of biological applications, the arising of high-throughput single-cell techniques has allowed microbiologists to follow heterogenous populations of isolated bacteria (where the structure is given by their length, biological markers or any other observable) through time. Thereby, this also grants access to the age structure and has put in evidence the non-trivial dependence of age and other observables at the individual and population scales. The most recent models of bacterial growing include then some sort of age variable, which might not correspond exactly with the \textit{chronological age}, but which might rather evolve in time as a function of the individual traits. { For example, in the \textit{adder model} of bacterial growth discussed in Example~\ref{ex:adder} below, the \textit{age} corresponds to the length added from the birth of the individual, so it grows along with the size variable, but resets at 0 at each reproduction event.} This \textit{age} variable still obeys a renewal equation, which justifies nonetheless its name. { In this regard, an important biological and mathematical question concerns the long-time behaviour of such dynamics: whether a certain steady-state distribution exists and the convergence rate towards it. Biologically, it allows to explain the observation of homeostatic behaviours in experimental timescales. Mathematically, it corresponds to the non trivial task of extending classical stability results to a broader class of stochastic models, by studying the spectral and ergodic properties of a certain family of operators. }  

In this spirit, we study the long-time behaviour of a stochastic process modelling non-conservative population dynamics which are formalised as a measure-valued process $(Z_t)_{t \geq 0}$ with values in the point measures over $\rr_+^2$, $\mathcal{M}_p(\rr_+^2)$,  which represents the age and size of the individuals. For every instant $t > 0$ we can write
\begin{equation}
Z_t = \sum_{i \leq \ap{Z_t}{1}} \delta_{\mathbf{x}_i},
\end{equation}
where $\mathbf{x}_i = (a_i, y_i)$ denotes the vector trait of individual $i$, consisting in its age $a_i$ and size $y_i$. {  We assume that each cell in the population behaves independently. The population then evolves in the continuous time through two fundamental dynamics: growth and division. Whilst growth is assumed to be deterministic, the division mechanism will account for the observed stochasticity. We present below informally the main characteristics of these two ingredients, which are formalised in Section~\ref{sec:pre}:
	
\begin{itemize}
		\item \textbf{Growth and ageing:} Between reproduction events, the variable $\mathbf{x}$ evolves following the deterministic ODE \[ \mathbf{x}'(t) = g(\mathbf{x}(t)) .\]
	The function $g:\rr_+^2 \to \rr_+^2$ represents the growth rate of the size and age coordinates. We denote $\bx \mapsto \varphi^t(\bx)$ the deterministic flow induced by the ODE with initial condition $\bx$ (see Lemma~\ref{lemma:flowprop} for the details), this is, the age and size at time $t$ of an individual of age and size $\bx$ at time 0. For example, if $a$ coincides with the chronological age, and $y$ grows exponentially at rate $\lambda$, we will have $g(a,y) = (1, \lambda y)$ and $\varphi^t(a,y) = (a+t,y e^{\lambda t})$. More interestingly, age and size can evolve in a dependent way. It is the case in Example~\ref{ex:adder} discussed further below: the \textit{adder model} of bacterial growth. There, the age corresponds to the size added since the last division, so that age and size grow at the same rate. Therefore, if $y$ grows exponentially at rate $\lambda$, we have $g(a,y) = (\lambda y, \lambda y)$ and $\varphi^t(a,y) = (a + y e^{\lambda t} - y,  y e^{\lambda t})$. 
		\item \textbf{Reproduction:} Individuals divide independently. An individual of birth state $\bx = (a_0, y_0)$ at time $t_0$ will reproduce at a random time $t_0 + T$, where $T$ is distributed according to
	\[
\mathbb{P}_{\bx}\pars{T \geq t} = \exp\pars{-\int_0^t \beta(\varphi^{s}(\bx)) ds}.
\]
The function $\beta : \rr_+^2 \to \rr_+$ is called the reproduction rate. The value of $\beta(a,y)$ gives the infinitesimal probability by unit of time for an individual of age $a$ and size $y$ to reproduce. 
Then, when an individual of state $\bx$ reproduces, it gives birth to new individuals of age 0 and random size $Z$, that depends on the value of $\bx$. The probability distribution of $Z$ conditional to $\bx$ is characterised by the transition kernel $k : \rr_+^2 \to \rr_+$, which is a positive integrable function. Thus, the number of new individuals of size $z$ produced by an individual of state $\bx$ is proportional to $ k(\bx, z) dz $. In particular, the value of the integral $\int_0^{+\infty} k(\bx, z) dz $ gives the total offspring produced by that individual. The age variable, on the other hand, resets at 0 at each jump. This means that the transition kernel over $\rr_+^2$ is degenerate, of the form $\mathbf{x} \mapsto \delta_{0} (da) \otimes k (\mathbf{x}, z) dz$.  \\
\end{itemize}
}

{ Following the approach introduced by~\cite{fm,tran}, by using a pathwise representation of $Z_t$ with respect to a Poisson point measure,} we can prove that for every $f \in C_b^{1,1}(\rr_+^2)$, $Z_t$ decomposes as a semi-martingale of the form
\begin{equation}
\ap{Z_t}{f} \overset{\textrm{def}}{=} \int_{\rr_+^2} f(\mathbf{x}) Z_t(d \mathbf{x}) = \ap{Z_0}{f} + \int_0^t \ap{Z_s}{\mathcal{Q}f} ds + \mathscr{M}_t^f ,
\label{eq:martingale}
\end{equation}
where $\mathscr{M}_t^f$ is a squared-integrable martingale, and $\mathcal{Q}$ is given for every $f~\in~C_b^{1,1}(\rr_+^2)$ by
\begin{align}
\mathcal{Q}f (\mathbf{x})
= g(\mathbf{x})^\top \nabla f(\mathbf{x})
+ \beta ( \mathbf{x})  \pars{ \int_0^\infty  f(0,z ) k(\mathbf{x}, z) dz   -f(\mathbf{x})}  \quad \forall \mathbf{x} \in \rr_+^2
\label{eq:Q}
\end{align}
 { In the following, we consider the extended version of the generator $\mathcal{Q}$ (see for example Section 20.3.2 of~\cite{mtbook}), associated to a domain $D(\mathcal{Q})$ in which the integral term is well defined. We recall that a function $f$ is said to be in the domain of the extended generator of $\mathcal{Q}$ if there exists a measurable function $u$ such that
\[
\pars{ \ap{Z_t}{f} - \ap{Z_0}{f} - \int_0^t \ap{Z_s}{u} ds}_t
\]
is a local martingale. In that case we will write $\mathcal{Q} f = u$. This is a natural definition since our starting point is the decomposition \eqref{eq:martingale}. 
}

{ \begin{example}[The adder model]
\label{ex:adder}
In the particular case of the bacterial proliferation model that interests us, and that will be studied in Section~\ref{sec:adder}, we consider the dynamics of an age-size-structured population of \textit{E. coli} bacteria as a measure-valued process with values in $\mathcal{M}_p(\cx)$, the point measures over the state space $\cx = \{ (a,y) \in \rr_+^2: 0 < a < y, y > 0 \}$, where $a$ represents the added size and $y$ the current size of each cell. This is, the age of a cell is given by the difference between its current size and its initial size. Therefore, the variable $a$ is not a chronological age, and has actually length units. However, it is indeed a variable that increases with time and is reset to zero after reproduction events. The importance of considering the added size as a structural variable to accurately model the growing dynamics of \textit{E. coli} has been strongly suggested in the recent years by experimental works and statistical analysis in unperturbed conditions~\cite{Taheri-Araghi2015,doumic}, but also in the case where the growth is perturbed by anti-DNA antibiotics (unpublished work by J. Broughton, M. El Karoui, S. M\'el\'eard and the author). 
The dynamics are driven by the generator
\begin{align}
\mathcal{Q}f (a,y)
=& \lambda y \pars{ \partial_a + \partial_y }f(a , y) \nonumber \\
&+ \lambda y B(a)  \pars{ 2 \int_0^1 f(0,\rho y)  F(\rho) d\rho  -f(a,y)}  - d_0 f(a,y). 
\label{eq:Q_adder}
\end{align}
In our previous notation this translates as $g(a,y) = (\lambda y, \lambda y)$, $\beta(a,y) = \lambda y B(a)$, and $k((a,y),z) = 2 \frac{1}{y} F\pars{\frac{z}{y}} \mathds{1}_{z \leq y}$, where $F$ has support in $[0,1]$. The growth dynamics correspond to an exponential elongation at constant rate $\lambda > 0$. The second term in $\mathcal{Q}$ represents the divisions, which occur at rate $\lambda y B(a)$ where $B$ is a hazard function such that for every individual,
\[
\mathbb{P}\pars{\textrm{Added size at division} \geq a} = \exp\pars{-\int_0^a B(s) ds}.
\]
Hence, the jump term reads as follows: a cell of size $y$ and added size $a$ divides at rate $\lambda y B(a)$, and is replaced by two cells of added size 0 and sizes $\rho y$ and $(1-\rho) y$ respectively, where $\rho$ is randomly distributed following the density $F$. The third term represents deaths at a constant rate $d_0>0$.  

In the following we will study the much general model generated by~\eqref{eq:Q}.
\end{example}
}

Our goal is to obtain the long-time behaviour of the first-moment semigroup $M_t f(\mathbf{x}) := \esper{\delta_{\mathbf{x}}}{\ap{Z_t}{f}}$, { which describes the expected behaviour of the population}. In particular, we prove a \textbf{Malthusian behaviour}:
\begin{equation}
M_t f(\mathbf{x}) = h(\mathbf{x}) e^{\lambda t} \ap{\pi}{f} + O \pars{ e^{(\lambda - \omega) t} },
\end{equation}
which shows the convergence of $e^{-\lambda t} M_t$ towards a unique stationary measure $\pi$ at an exponential rate. The parameter $\lambda > 0$ is called the Malthus parameter and represents the growth rate of the population, so that $e^{-\lambda t}$ allows to rescale the mean population size as $t \to +\infty$. The function $h$ propagates the effect of the initial structure of the population. The constant $\omega$ indicates the convergence rate towards $\pi$.  

Different methods have been developed during the recent years to prove this behaviour: spectral methods, as reviewed in~\cite{Mischler2016} (see for example~\cite{Roget2017} for an application to a close model); others based on the study of the associated semigroup by Harris' theorem as proposed in some general frameworks by~\cite{Bansaye2019harris,Canizo2020,Bansaye2020Doeb} with recent applications in the models considered by~\cite{Bouguet,Tomasevic2020,cloez2020longtime}. We will follow the latter methods, using the criteria established by Meyn and Tweedie~\cite{Meyn1993iii}, namely: a petite-set condition (H1) and the existence of a Lyapunov function (H2), as given in Theorem~\ref{thm:harris}. This methods present an alternative to PDE techniques, where criteria based on the probabilistic control of moments replace the harder to obtain Poincaré-type inequalities. 

We explore two directions left open in the previous applications, which represent also the sources of our major technical issues: first, the bi-dimensionality of the dynamics, and second, the degeneracy of the transition kernel. Indeed, the underlying stochastic process consists on unidimensional trajectories over a two-dimensional space. Hence, to uniformly bound in probability the region explored by these trajectories with respect to a non-degenerate measure is not trivial. Similar difficulties have been found for other two-dimensional models such as~\cite{Fonte2021LongTB,Torres2021AMT,Costa2016}. Here, we propose to construct explicit trajectories and to average them in time with respect to a nice sampling measure. {The inclusion of time sampling allows to compensate the lack of stochasticity of the degenerate jump-transport dynamics on an unbounded state space. Indeed, it is not trivial to find a fixed time $t_0 > 0$ such that the trajectories originated from any initial state mix uniformly on the support of some non-trivial measure of the unbounded two-dimensional space $\mathcal{X}$. However, if we authorise the time $t_0$ to be sampled from some probability distribution, chances are the uniform exploration of the space will be easier to prove, as we show indeed later. More technically, the utilisation of a petite-set condition instead of a small-sets one is key to obtain the convergence in this setting. }

Moreover, compared to the previous works mentioned above, the probabilistic framework brings naturally to work with the operator $\mathcal{Q}$ instead of its dual, as in the more classical PDE settings. Thus, this work lies also in the framework of measure solutions as rigorously developed for example in~\cite{Gabriel2018} for the one-dimensional conservative case. Moreover, only the existence of eigenelements for $\mathcal{Q}$ is needed to be able to compute the Doob $h$-transform and use Harris' theorem. Then, the existence of the \textit{direct} eigenfunction associated to the classical PDE is a consequence of our main result. Our method is then in the spirit of~\cite{Canizo2020}, where the authors could exploit known results of existence of the dual eigenelements in the one-dimensional case provided by~\cite{BCG13,DoumicJauffret2010EigenelementsOA}. In our case, we will have to adapt the latter to the two-dimensional degenerate case studied here.

In particular, we will apply our method to determine the exponential convergence towards a stable size distribution in a bacterial proliferation model called the \textit{adder} model~\cite{Taheri-Araghi2015,Lin2020,Hall1991SteadySD,gabriel:hal-01742140}. Individual cells are structured by their added size $a$ which renews to 0 at each division, and their size $y$ which evolves deterministically at exponential rate. The existence of a steady-state distribution and its form was already known since~\cite{Hall1991SteadySD}, however the exponential convergence could not be obtained using entropy methods by~\cite{gabriel:hal-01742140}. Since the eigenelements of the generator are known in this case, by the direct application of Theorem~\ref{thm:harris}, our method permits to obtain the exponential convergence while evading technical issues linked to the lack of compactness of the model, which make a classic treatment by PDE and hypocoercivity methods harder to prove and less general.

{
Finally, it is worth noticing that other models share similar dynamics with the ones generated by \eqref{eq:Q}. In an unrelated context, but fairly similar setting, Piecewise Deterministic Markov Processes (PDMP) have been recently used to sample target distributions in the framework of Markov Chain Monte Carlo (MCMC) methods, as described for example in~\cite{mcpdmp}. There, an important task is to show good convergence rates of the MC-PDMP towards the target stationary distribution.  Methods relying on drift-minorisation conditions, similar to conditions (H1) and (H2) of Theorem~\ref{thm:harris}, have been proven useful in that context~\cite{zigzag,bouncy,fetique}, where issues related to dimensionality and degeneracy might also be encountered. Notice however that the processes considered in all these contexts are always conservative.}

{Other biological models can also be generated by similar semigroups, so their spectral and ergodic properties might be deduced from our results. For example, the non-trivial uniform bound estimate for the population growth rate established in Step~5 of the Proof of Proposition~\ref{prop:eigenelements} relies mainly in the assumption that the newborn sizes are almost surely smaller than the parent size, without any additional requirements for the form of the kernel $k$. In particular, we do not require conservation of mass $y = \int k((a,y),z) dz $, as in classical size-structured models~\cite{DoumicJauffret2010EigenelementsOA}. Therefore, the same arguments can be used in general age-trait models that authorise only negative jumps for the trait coordinate. Biologically, this could account for a trait evolving deterministically, and which is almost surely eroded or corrupted at each reproduction event. This is the case, for example, in some telomere-shortening models~\cite{telomere}. In particular, Assumptions~\ref{ass:1} give necessary conditions such that the growth rate $g$ compensates the \textit{fragmentation} events arriving at each reproduction, in order to preserve ergodicity.}

\section{Malthusian behaviour}

We are interested in the average dynamics as given by first-moment semigroup $M_t$ defined for every test function $f \in \mathcal{C}_b^{1,1}(\rr_+^2)$ by:
\begin{equation}
M_t f(\mathbf{x}) :=  \espero{\ap{Z_t}{f} | Z_0 = \delta_{\mathbf{x}}}  \quad \forall \mathbf{x} \in \rr_+^2
\end{equation}
Using Markov's property it's easy to see that $M_t$ verifies the semigroup property. However it is not a Markovian semigroup since it does not necessarily preserve mass (we say it is non conservative). Moreover, using the semi-martingale decomposition \eqref{eq:martingale}, we verify that $M_t$ is the semigroup associated to the extended generator $\mathcal{Q}$. This is, for every test function $f \in C^{1,1}(\cx)$, it is the weak solution of Kolmogorov's equations
\begin{equation}
\partial_t M_t f = M_t \mathcal{Q} f = \mathcal{Q} M_t f.
\end{equation}
Moreover, for any finite measure $\mu$ we define the dual semigroup as the measure $\nu M_t$ given by:
\begin{align*}
(\nu M_t) f := \nu(M_tf) = \int_{\cx} M_t f(\mathbf{x}) \nu(d{\bx})
\end{align*}
So by definition we have $(\mu M_t) f = \mu (M_t f) $ which we write as $\mu M_t f$. 

Our main result states the Malthusian behaviour of the semigroup by means of Harris' Ergodic Theorem as stated in Theorem 6.1 of~\cite{Meyn1993iii}, which we recall below in Theorem~\ref{thm:harris}. 

\begin{theorem}[$V$-uniform Ergodic Theorem (also known as Harris' Ergodic Theorem) (Theorem 6.1 of~\cite{Meyn1993iii})]
	\label{thm:harris}
	Let $(X_t)_t$ be a right-continuous Markov process with values in some locally compact separable metric space $E$ equipped with its Borelian set $\mathcal{B}(E)$, and let $\mathcal{A}$ be the infinitesimal generator of $X$. We call $P_t$ the associated transition semigroup. If the two following conditions are verified:
	\begin{enumerate}[label=(H\arabic*)]
		\item \textbf{Minorisation condition for compact sets}. 
		
		All compact sets of $E$ are petite for a skeleton chain of $X$. This is, for every compact set $\mathscr{K} \subset E$ there's a probability mass distribution $\mu = (\mu_n)_{n \in \mathbb{N}}$ over $\mathbb{N}$ and some $\Delta > 0$ such that there exists a non-trivial measure $\nu$ (which might depend on $\Delta$ and $\mu$) over $\mathcal{B}(E)$ that for every $\bx \in \mathscr{K} $ gives the following lower bound:
		\[
		\ap{\mu}{\delta_{\bx}P_{\cdot}f} = \sum_{n \in \mathbb{N}} \mu_n P_{n \Delta} f(\mathbf{x}) \geq \ap{\nu}{f} . \]					
		\item \textbf{Foster-Lyapunov drift condition.}
		
		There exists a coercive function $V$, { meaning that $V(\bx) \to +\infty$ as $ \norm{\mathbf{x}} \to +\infty$, such that $V(\bx) \geq 1$ for all $\bx$,} and there exist some $c >0$, $d < \infty$ such that
		\[
		\mathcal{A} V(\mathbf{x}) \leq - c V(\mathbf{x}) + d  \quad \forall \bx \in E
		, \]
	\end{enumerate}
	Then, there exist a unique non-trivial probability measure $\pi$ and $C, \omega > 0$ such that for every $\bx \in E$ and $t \geq 0$
	\begin{equation}
	\norm{\delta_\bx P_t - \pi}_{V} \leq C V(\mathbf{x}) \exp(-\omega t),
	\end{equation}
	{  where the $V$-norm defined by
	\[
	\norm{\mu}_{V} := \sup_{g : \norm{g} \leq V } \abs{\ap{\mu}{g}}
	\]
	is an extension of the total variation norm. In particular $\norm{\mu}_1 = \norm{\mu}_{\textrm{TV}}$.  }
\end{theorem}

\begin{remark}
{
    The reader familiar with other versions of $V$-uniform ergodicity theorems, such as Theorem 20.3.2 of~\cite{mtbook}, the results of~\cite{down}, or more the more recently derived version of~\cite{hairermattingly} might find that conditions (H1) and (H2) are written in a slightly exotic way, even thought they are extracted without much modification from source~\cite{Meyn1993iii}. We address briefly these potential concerns. First, it is worth noticing that petiteness condition (H1) is stressed for a skeleton chain on the process and not for the continuous-time process. This allows to circumvent issues related to periodicity. A classical pathological example is the clock process, defined by the deterministic semigroup $P_t f(x) = f(x \textrm{e}^{2 \pi i t } )$, for $x \in \mathbb{S}_1 = \{ z \in \mathbb{C} : |z| = 1\}$, $t \geq 0$. It consists on periodic orbits along $\mathbb{S}_1$. Since even irrational skeleton chains are not mixing, it is not possible to establish a uniform minorisation condition valid for all starting points of any fixed compact set $\mathscr{K}$ of $\mathbb{S}_1$. Condition (H1) is therefore not verified by the clock process. Notice, however, that for the uniform sampling measure $\mu(dt) = \mathds{1}_{[0,1]}(t)dt$, the continuous-time semigroup $P_t$ does verify Doeblin condition $\ap{\mu}{\delta_{x}P_{\cdot}} \geq \nu$, for all $x \in \mathbb{S}_1$, with $\nu$ the uniform measure over $\mathbb{S}_1$. Thus the importance of testing petiteness for the skeleton chains.  
    } 
    
    {
    Second, condition (H2) is usually stated with an indicator function over some petite set $C$, this is, as
    \begin{equation}
\mathcal{A} V \leq - c V + b \mathds{1}_C,
\tag{V4}
\end{equation}
called drift condition (V4) in~\cite{mtbook} and many later works. Indeed, from Section 5 of~\cite{down}, it can be shown that if the function $V$ is \textit{unbounded off petite sets}, i.e., if for every $n \in \mathbb{N}$, the set $\setof{\mathbf{x} \in \cx : V(\bx) \leq n}$ is either empty or petite, condition (V4) is equivalent to (H2). In our case, since $V$ is coercive and that by (H1), all compact sets are petite for some skeleton chain, we have that $V$ is unbounded off petite sets for that skeleton chain, and therefore an equivalent discrete-time version of (V4) is verified for the skeleton (condition $(\mathscr{D}_{T})$ of~\cite{down}, p. 1679). Theorem 5.1 of~\cite{down} shows finally that $(\mathscr{D}_{T})$ and (V4) (called $(\tilde{\mathscr{D}})$ therein) are actually equivalent.}
\end{remark}

Notice that we need a Markovian (conservative) semigroup. To overcome this problem, similarly as in~\cite{Canizo2020}, we perform a so-called Doob $h$-transform, to obtain a conservative semigroup $P_t$ with the dynamics of $M_t$. To do so, we require first to have some pair $(\lambda, h)$ such that $\mathcal{Q}h = \lambda h$ and $h>0$. Then, using such pair we define
\begin{equation}
P_t f(\mathbf{x}) := \frac{M_t (hf)(\mathbf{x})}{e^{\lambda t} h(\mathbf{x})}.
\label{eq:Pt}
\end{equation}
Then we can come back the ergodic behaviour of $(M_t)_{t \geq 0}$ by looking at the limit of $M_t f = e^{\Lambda t} h P_t \pars{f/h}$. In particular, the generator associated with $P_t$ is given explicitly by Eq. \eqref{eq:A}.

\begin{proposition} Suppose the existence of a pair $(\lambda, h)$, $\lambda >0$, $h >0$ such that $\mathcal{Q}h = \lambda h$. Then, $P_t$ defined by Eq. \eqref{eq:Pt} is a positive Markovian semigroup whose infinitesimal generator is given for $f \in C_b^{1,1}(\rr_+^2)$ by 
	\begin{align}
	\mathcal{A}f (\mathbf{x})
	=  g(\mathbf{x})^\top \nabla f(\mathbf{x})
	+ \beta ( \mathbf{x})  \pars{ \int_0^\infty   \sqkets{ f(0,z) - f(\mathbf{x}) } \frac{h(0, z )}{h(\mathbf{x})} k (\mathbf{x}, z) dz  }  \quad \forall \mathbf{x} \in \rr_+^2
	\label{eq:A}
	\end{align}
\end{proposition}
\begin{proof}
	By definition and evaluating at $t=0$ we have:
	\begin{align*}
	\mathcal{A}f(\mathbf{x}) &= \left. \frac{\partial}{\partial t} P_t f \right|_{t=0} (\mathbf{x}) \\
	&=\left. \frac{M_t \mathcal{Q}(hf)}{e^{\lambda t}h} - \frac{\lambda M_t(hf)}{e^{\lambda t}h}    \right|_{t=0} (\mathbf{x}) \\
	&= \frac{\mathcal{Q}(hf)(\mathbf{x})}{h(\mathbf{x})} -\lambda f(\mathbf{x}) 
	\end{align*}
	Then, using the value of $\mathcal{Q}$ applied to $hf$ and that $\mathcal{Q} h = \lambda h$ we get
	\begin{align*}
	\frac{\mathcal{Q} (hf) (\bx)}{h(\bx)} &= \frac{f(\bx)}{h(\bx)} g(\mathbf{x})^\top \nabla h(\mathbf{x}) + g(\mathbf{x})^\top \nabla f(\mathbf{x}) 
	+ \beta ( \mathbf{x})  \pars{ \int_0^\infty  h(0,z) f(0,z ) \frac{k(\mathbf{x}, z)}{h(\bx)} dz   - f(\mathbf{x})} \\
	&= \frac{f(\bx)}{h(\bx)}  \pars { g(\mathbf{x})^\top \nabla h(\mathbf{x}) - \beta(\bx) h(\bx)} +  g(\mathbf{x})^\top \nabla f(\mathbf{x})  \\
	& \qquad + \beta(\bx) \int_0^\infty  h(0,z) f(0,z ) \frac{k(\mathbf{x}, z)}{h(\bx)} dz \\
	&=\frac{f(\bx)}{h(\bx)} \pars { \lambda h(\bx) - \beta(\bx) \int_0^\infty  h(0,z)  k(\mathbf{x}, z) } + g(\mathbf{x})^\top \nabla f(\mathbf{x})  \\
	& \qquad + \beta(\bx) \int_0^\infty  h(0,z) f(0,z ) \frac{k(\mathbf{x}, z)}{h(\bx)} dz \\
	&= \lambda f(\bx) + g(\mathbf{x})^\top \nabla f(\mathbf{x})  + \beta(\bx) \int_0^\infty \sqkets{f(0,z ) - f(\bx)} \frac{h(0,z)}{h(\bx)}k(\mathbf{x}, z) dz
	\end{align*}
	Finally, subtracting $\lambda f(\bx)$ we obtain the form of generator $\mathcal{A}$.  
\end{proof}

Hence, the work is structured as follows: first, in Section 3 we prove the existence of a pair $(\lambda, h)$ which solves the eigenvalue problem $\mathcal{Q} h = \lambda h$ under the Assumptions~\ref{ass:all}. The same set of assumptions allows us to prove the Doeblin condition (H1) in Section 4. We do not provide a general Foster-Lyapunov condition (H2), suitable for our general case. However, we show its existence in our application to a growth-fragmentation model in Section 5.  This last model has already been studied since the works of~\cite{Hall1991SteadySD}, and the exponential convergence has been recently shown in~\cite{gabriel:hal-01742140} using Generalized Relative Entropy techniques. Here, we show that the knowledge of the eigenelements $(\lambda, h)$ for the generator allows to provide a simpler proof of convergence using Harris' theorem. Indeed, the arguments presented in Section~\ref{sec:eigenelementsQ} can be avoided when the existence of eigenelements is known apriori, which might be the case in several practical applications. Nonetheless, our general method allows us to give an answer to one of the perspectives listed by~\cite{gabriel:hal-01742140}, who couldn't generalise their argument in the case of a general drift function $g$. Thus, our main result reads as follows:

\begin{theorem}[Exponential ergodicity]
	Under Assumptions~\ref{ass:all} and if the Lyapunov-Foster condition (H2) of Theorem~\ref{thm:harris} is verified for some coercive function $V:\rr_+^2 \to \rr_+$, there is a unique probability measure $\pi$ such that there exist constants $C, \omega, \Lambda > 0$ which verify for every initial condition $\mu_0 \in \mathcal{M}_p(\rr_+^2)$ 
	\begin{equation}
	\norm{e^{-\Lambda t} \mu_0 M_t - \ap{\mu_0}{h}\pi}_{V} \leq C \ap{\mu_0}{V} e^{-\omega t}.
	\label{eq:exponCV}
	\end{equation}
	Moreover, $\pi$ is absolutely continuous with respect to the Lebesgue measure. 
	\label{thm:main}
\end{theorem}

\section{Preliminary definitions and assumptions}
\label{sec:pre}

We begin by recalling some useful properties of the deterministic flow, which are classical results for an autonomous system of first order ODE (refer for example to Theorem D.1 of~\cite{lee2003introduction}):

\begin{lemma}[Flow properties and notations.]
	Let $\mathbf{x} \in \rr_+^2$. Consider $g = (g_1,g_2) \in C^1(\rr_+^2)$ and suppose that $g_1 > 0$. The autonomous first-order system of Ordinary Differential Equations (ODE)
	\begin{equation}
	\begin{split}
	\frac{d \mathbf{u}(t)}{dt} &= g \pars{
		\mathbf{u}(t)} \ , \ t \in \rr \\
	\mathbf{u}(0) &= \mathbf{x}
	\end{split}
	\label{eq:ODE}
	\end{equation}
	defines a unique flow $
	\varphi^t : \cx \ni \mathbf{x} \mapsto \varphi^t \pars{\bx} \in \cx$ which is the solution $\mathbf{u}(t)$ of \eqref{eq:ODE} at time $t$ with initial condition $\mathbf{x} \in \cx$ where $\cx = \bigcup_{y \geq 0} \Gamma_{(0,y)}^+$ where $\Gamma_{\bx}^+$ will be defined below. We write $\varphi^t = (\varphi^t_1, \varphi^t_2)$ for the marginal flows of the age and size. We define then $\Gamma_{\mathbf{x}}^+ = \setof{\varphi^t (\mathbf{x}) , t \geq 0}$ and $\Gamma_{\mathbf{x}}^- = \setof{\varphi^t (\mathbf{x}) , t \leq 0}$ and call $\Gamma_{\mathbf{x}} = \Gamma_{\mathbf{x}}^+ \cup \Gamma_{\mathbf{x}}^-$ the unique orbit passing through $\mathbf{x}$.  Moreover:
	\begin{enumerate}
		\item The flow is a group in the time variable: $\varphi^t \varphi^s = \varphi^{t+s} = \varphi^s \varphi^t $, $\varphi^0 = \mathrm{Id}$, and has inverse $\pars{\varphi^{t}}^{-1} = \varphi^{-t}$, which is the solution to the ODE $\mathbf{u}'(t) = -g \pars{
			\mathbf{u}(t)}$.
		\item The flow depends smoothly on the initial conditions: $\forall t \in \rr$, $\varphi^t \in C^1(\rr_+^2)$. We call $
		\mathcal{D}\varphi^t (\mathbf{x}) $ the Jacobian matrix of the flow with respect to the initial condition.
		\item For all fixed $\mathbf{x} = (a_0,y_0) \in \cx$, if $g_1 > 0$, then there is a unique function $Y_\mathbf{x} : \rr_+ \to \rr_+$ such that for all $(a, y) \in \Gamma_{\mathbf{x}}$, we have $Y_{\mathbf{x}}(a) = y$. This represents the size at a given age of an individual with initial condition $\mathbf{x}$. In other words, for all $t \geq 0$,
		\[
		\varphi^t \pars{\bx} = \pars{a(t), Y_\bx (a(t))}.
		\]
		Moreover, $Y_{\mathbf{x}} \in C^1(\rr_+)$ and it is solution of the first order one-dimensional ODE 
		\[Y_{\mathbf{x}}'(a) = \frac{g_2(a,Y_{\mathbf{x}}(a))}{g_1(a,Y_{\mathbf{x}}(a))} \quad ; \quad Y_{\mathbf{x}}(a_0) = y_0 \]
		Analogously, one defines its inverse function $A_{\mathbf{x}}(y)$ which gives the age at size $y$ for an individual with initial condition $\mathbf{x}$, and hence verifies
		\[
		\varphi^t \pars{\bx} = \pars{A_\bx(y(t)), y(t)} , \quad t \geq 0.
		\]			
		\item For all fixed $\mathbf{x} \in \cx$, we write $\phi_{\mathbf{x}}(t) := \varphi^t(\mathbf{x})$ as a function of time (from $\rr$ to $\rr_+^2$). Then, the inverse function $\phi_{\mathbf{x}}^{-1} : \Gamma_\bx \to \rr$ such that $\phi_{\mathbf{x}}^{-1}(\phi_{\mathbf{x}}(t)) = t$ is well defined. For every $\bx_0 \in \cx$ and $\bx_1 \in \Gamma_{\bx_0}$ we read $\phi_{\bx_0}^{-1}(\bx_1)$ as \textit{the time needed along} $\Gamma_{\bx_0}$ \textit{to go from} $\bx_0$ \textit{to} $\bx_1$.			
		Moreover if we write $\bx_0 = (a_0,y_0), \ \bx_1 = (a_1,y_1)$, this quantity is given by
		\[
		\phi_{\bx_0}^{-1}(\bx_1) = \int_{a_0}^{a_1} \frac{1}{g_1 \pars{a, Y_{\bx_0}(a)}  } da = \int_{y_0}^{y_1} \frac{1}{g_2 \pars{A_{\bx_0}(y), y}  } dy.
		\]
		Importantly, for the set of assumptions given below, we have $ 0 < \phi_{\bx_0}^{-1}(\bx_1) < \infty$ for all $\bx_0 \in \cx \setminus \setof{0}$ and $\bx_1 \in \Gamma_{\bx_0}$.
	\end{enumerate}
	\label{lemma:flowprop}
\end{lemma}		
Let us also consider the following probability space which well be useful to compute and interpret some of the estimates which will be obtained below.
\begin{definition}
	Consider a probability space $(\rr_+,  \mathcal{B}(\rr_+), \mathbb{P}_{\bx})$ in which the random couple $(T,Z) \in \rr_+ \times \rr_+$ gives the first jump time $T$ and size $Z$ after this first jump of a trajectory beginning at $\bx \in \cx$. Hence, for all $\bx \in \cx$, the couple $(T,Z)$ has joint probability density
	\[
	p_{\bx}(t,z) = \frac{1}{C_\bx} k (\varphi^t (\bx),z ) \psi(t\vert \bx) ,
	\]
	where the normalisation constant is given by
	\[
	C_\bx = \int_0^\infty \int_0^\infty k (\varphi^t (\bx),z' ) \psi(t\vert \bx) dt dz',
	\]
	which is the mean number of offspring produced by an individual of initial configuration $\bx$ after its first jump, and
	\[
	\psi(t\vert \bx) = \beta(\varphi^t(\bx)) \exp \pars{ - \int_0^t \beta \pars{\varphi^s(\bx) } ds  }
	\]
	is the marginal probability density of the time of the first jump, conditionally to the initial configuration $\bx$, and which is well defined for the set of assumptions given below. We write $\mathbb{E}_{\bx}$ the associated expectation. Fig.~\ref{fig:defs} summarises the definitions introduced in this section.
	\label{def:Px}
\end{definition}	

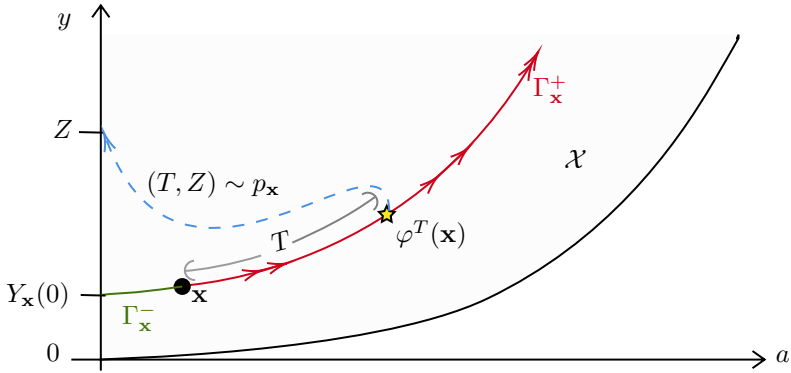
\begin{figure}[ht]
	\centering

	\tikzset{every picture/.style={line width=0.75pt}} 
	
	\begin{tikzpicture}[x=0.65pt,y=0.65pt,yscale=-1,xscale=1]
	
	\draw [color={rgb, 255:red, 0; green, 0; blue, 0 }  ,draw opacity=1 ][fill=gray!10 ,fill opacity=0.2]   (141,209.96) .. controls (136.2,209.56) and (296,207.39) .. (364,174.39) .. controls (432,141.39) and (471.41,87.99) .. (495.04,46.99) .. controls (518.68,5.99) and (504.58,29.66) .. (509,20.39) ;
	\draw  [draw opacity=0][fill=gray!10  ,fill opacity=0.2 ] (509,20.39) -- (141,209.96) -- (140.93,20.52) -- cycle ;
	\draw  (124,209.96) -- (524,209.96)(141,3.39) -- (141,216.39) (517,204.96) -- (524,209.96) -- (517,214.96) (136,10.39) -- (141,3.39) -- (146,10.39)  ;
	\draw    (128.41,77.29) -- (142.74,77.63) ;
	\draw [color={rgb, 255:red, 208; green, 2; blue, 27 }  ,draw opacity=1 ]   (183.67,167.33) .. controls (177.4,166.8) and (317,168) .. (393.4,30.4) ;
	\draw [shift={(393.4,30.4)}, rotate = 119.04] [color={rgb, 255:red, 208; green, 2; blue, 27 }  ,draw opacity=1 ][line width=0.75]    (17.64,-3.29) .. controls (13.66,-1.4) and (10.02,-0.3) .. (6.71,0) .. controls (10.02,0.3) and (13.66,1.4) .. (17.64,3.29)(10.93,-3.29) .. controls (6.95,-1.4) and (3.31,-0.3) .. (0,0) .. controls (3.31,0.3) and (6.95,1.4) .. (10.93,3.29)   ;
	\draw [color={rgb, 255:red, 128; green, 128; blue, 128 }  ,draw opacity=1 ]   (251.5,141.75) .. controls (267,134.75) and (283.5,126.75) .. (299.5,114.89) ;
	\draw [shift={(299.5,114.89)}, rotate = 143.45] [color={rgb, 255:red, 128; green, 128; blue, 128 }  ,draw opacity=1 ][line width=0.75]      (5.59,-5.59) .. controls (2.5,-5.59) and (0,-3.09) .. (0,0) .. controls (0,3.09) and (2.5,5.59) .. (5.59,5.59) ;
	\draw [color={rgb, 255:red, 155; green, 155; blue, 155 }  ,draw opacity=1 ]   (233.5,149.25) .. controls (216.83,154.58) and (198.5,157.75) .. (189.5,158.25) ;
	\draw [shift={(189.5,158.25)}, rotate = 356.82] [color={rgb, 255:red, 155; green, 155; blue, 155 }  ,draw opacity=1 ][line width=0.75]      (5.59,-5.59) .. controls (2.5,-5.59) and (0,-3.09) .. (0,0) .. controls (0,3.09) and (2.5,5.59) .. (5.59,5.59) ;
	\draw  [fill={rgb, 255:red, 248; green, 231; blue, 28 }  ,fill opacity=1 ] (306,120.31) -- (307.4,123.36) -- (310.52,123.85) -- (308.26,126.23) -- (308.79,129.58) -- (306,128) -- (303.21,129.58) -- (303.74,126.23) -- (301.48,123.85) -- (304.6,123.36) -- cycle ;
	\draw  [fill={rgb, 255:red, 0; green, 0; blue, 0 }  ,fill opacity=1 ] (183.67,167.33) .. controls (183.67,164.94) and (185.61,163) .. (188,163) .. controls (190.39,163) and (192.33,164.94) .. (192.33,167.33) .. controls (192.33,169.73) and (190.39,171.67) .. (188,171.67) .. controls (185.61,171.67) and (183.67,169.73) .. (183.67,167.33) -- cycle ;
	\draw [color={rgb, 255:red, 208; green, 2; blue, 27 }  ,draw opacity=1 ]   (351.55,87.95) -- (335.45,103.25) ;
	\draw [shift={(335.45,103.25)}, rotate = 136.46] [color={rgb, 255:red, 208; green, 2; blue, 27 }  ,draw opacity=1 ][line width=0.75]    (10.93,-3.29) .. controls (6.95,-1.4) and (3.31,-0.3) .. (0,0) .. controls (3.31,0.3) and (6.95,1.4) .. (10.93,3.29)   ;
	\draw [shift={(353,86.57)}, rotate = 136.46] [color={rgb, 255:red, 208; green, 2; blue, 27 }  ,draw opacity=1 ][line width=0.75]    (10.93,-3.29) .. controls (6.95,-1.4) and (3.31,-0.3) .. (0,0) .. controls (3.31,0.3) and (6.95,1.4) .. (10.93,3.29)   ;
	\draw [color={rgb, 255:red, 208; green, 2; blue, 27 }  ,draw opacity=1 ]   (246.41,154.37) -- (233.76,157.93) ;
	\draw [shift={(233.76,157.93)}, rotate = 164.29] [color={rgb, 255:red, 208; green, 2; blue, 27 }  ,draw opacity=1 ][line width=0.75]    (10.93,-3.29) .. controls (6.95,-1.4) and (3.31,-0.3) .. (0,0) .. controls (3.31,0.3) and (6.95,1.4) .. (10.93,3.29)   ;
	\draw [shift={(248.33,153.83)}, rotate = 164.29] [color={rgb, 255:red, 208; green, 2; blue, 27 }  ,draw opacity=1 ][line width=0.75]    (10.93,-3.29) .. controls (6.95,-1.4) and (3.31,-0.3) .. (0,0) .. controls (3.31,0.3) and (6.95,1.4) .. (10.93,3.29)   ;
	\draw [color={rgb, 255:red, 74; green, 144; blue, 226 }  ,draw opacity=1 ] [dash pattern={on 4.5pt off 4.5pt}]  (307.4,123.36) .. controls (307,66.68) and (187.21,203.02) .. (143.4,79.51) ;
	\draw [shift={(142.74,77.63)}, rotate = 71.16] [color={rgb, 255:red, 74; green, 144; blue, 226 }  ,draw opacity=1 ][line width=0.75]    (10.93,-3.29) .. controls (6.95,-1.4) and (3.31,-0.3) .. (0,0) .. controls (3.31,0.3) and (6.95,1.4) .. (10.93,3.29)   ;
	\draw [color={rgb, 255:red, 65; green, 117; blue, 5 }  ,draw opacity=1 ]   (140.67,172.33) .. controls (157,170.39) and (153,172.39) .. (188,167.33) ;
	\draw    (129.41,172.29) -- (143.74,172.63) ;
	
	\draw (191.67,167.73) node [anchor=north west][inner sep=0.75pt]  [font=\normalsize]  {$\mathbf{x}$};
	\draw (388.67,42.4) node [anchor=north west][inner sep=0.75pt]  [color={rgb, 255:red, 208; green, 2; blue, 27 }  ,opacity=1 ]  {$\Gamma _{\mathbf{x}}^{+}$};
	\draw (108,199.4) node [anchor=north west][inner sep=0.75pt]    {$0$};
	\draw (111.67,68.23) node [anchor=north west][inner sep=0.75pt]    {$Z$};
	\draw (528.67,203.07) node [anchor=north west][inner sep=0.75pt]    {$a$};
	\draw (114.33,6.4) node [anchor=north west][inner sep=0.75pt]    {$y$};
	\draw (236.07,136.19) node [anchor=north west][inner sep=0.75pt]  [color={rgb, 255:red, 0; green, 0; blue, 0 }  ,opacity=1 ,rotate=-344.11,xslant=-0.05]  {$T$};
	\draw (309.4,126.76) node [anchor=north west][inner sep=0.75pt]    {$\varphi ^{T}(\mathbf{x})$};
	\draw (406.67,86.23) node [anchor=north west][inner sep=0.75pt]    {$\mathcal{X}$};
	\draw (165.67,98.23) node [anchor=north west][inner sep=0.75pt]    {$( T,Z) \sim p_{\mathbf{x}}$};
	\draw (151.67,175.4) node [anchor=north west][inner sep=0.75pt]  [color={rgb, 255:red, 65; green, 117; blue, 5 }  ,opacity=1 ]  {$\Gamma _{\mathbf{x}}^{-}$};
	\draw (84.67,163.23) node [anchor=north west][inner sep=0.75pt]    {$Y_{\mathbf{x}}( 0)$};

	\end{tikzpicture}
	\caption{Flow notations introduced in Lemma~\ref{lemma:flowprop} and the probabilistic definition of the random couple $(T,Z)$ introduced in Definition~\ref{def:Px}.}
	\label{fig:defs}
\end{figure}

Now, let us consider the following set of assumptions, whose biological meaning and implications are commented below.	
\begin{assumptions}
	Assume that we have
	\begin{enumerate}[(i)]
		\item \textit{Smooth and uniformly controlled flow:} $g = (g_1,g_2) \in C^1(\rr_+^2)$, $g_1 > 0$ and there are some constants $c_0, c_1, c_2 > 0$ such that for all $(a,y) \in \rr_+^2$
		\begin{align*}g_1(a,y) \geq c_0 \ a,  \quad g_1(a,y) \leq c_1(1+a), \quad g_2(a,y) \leq c_1(1+y), \\
		\vert  \partial_a g_i(a,y) \vert  \leq c_2(1+a+y), \quad  \vert  \partial_y g_i(a,y) \vert  \leq c_2(1+a+y).
		\end{align*}
		and for all $y>0$, $a\geq 0$, we have $g_2(a,y) \leq g_2(0,y)$.
		\item \textit{Regular reproduction rate:} $\beta \in C(\rr_+^2,\rr_+)$, and $B = \beta/g_1 \in C(\rr_+^2)$, such that there are constants $a^*, \beta_-,\beta_+ >0$ s.t. for all $a > a^* y \geq 0$,  $\beta_- < B(a,y) < \beta_+$, and $B(a,y) = 0$ for all $a \leq a^*$. 
		\item \textit{Regular transition kernel:} For all $z \geq 0$, $\mathbf{x} \mapsto k(\mathbf{x},z)$ is a continuous function on $\rr_+^2$,  and for all $\mathbf{x} \in \rr_+^2$, $z \mapsto k(\mathbf{x},z)$ is a continuous function on $\rr_+$. The total offspring of individuals of trait $\mathbf{x}$ is $\norm{k(\bx,\cdot) }_{1}:= \int_0^{+\infty} k(\mathbf{x}, z) dz $ with  $1 < \norm{k(\bx,\cdot) }_{1} \leq \bar K$ for all $\bx \in \rr_+^2$. In particular, we consider two distinct cases:
		\begin{enumerate}
			\item \textit{Fragmentation kernel:} For all $a\geq 0$, $\textrm{supp } k(a,y,\cdot) \subseteq (0,y)$.
			\item \textit{Compactly supported mutational kernel:} It exists a compact set $S \subset \rr_+$ such that for all $a\geq 0$, $\textrm{supp } k(a,y,\cdot) \subseteq S$, and some interval $I \subset \rr_+$ and $\epsilon_0 > 0$ such that for all $y \in S$ and $z \in S \cap \mathbb{B}_{\epsilon_0}(y)$, the open ball of radius $\epsilon_0$ around $y$, we have $I \subset \setof{a >0 :  \beta(a, y) k(a, y, z)  > 0}$.
		\end{enumerate}
		\item \textit{Lower bounded transition kernel:} For all fixed value of $z > 0$, there exists some non-empty open interval $D(z)$ with length bounded between $\delta_-$ and $\delta_+$, both independent of $z$, and a positive value $\varepsilon(z)$ such that for all $\mathbf{x} \in \rr_+^2$, $k(\mathbf{x}, z) > \varepsilon(z) \mathds{1}_{D(z)} (\mathbf{x})$. 
	\end{enumerate}
	\label{ass:all}
\end{assumptions}	
We comment on the meaning of these assumptions. Assumption~\ref{ass:all}-(i) ensures that the size and age do not explode in finite time. The control on the derivatives will also allow to control the influence of the initial conditions on the flow (Lemma 3.1.2). Assumption~\ref{ass:all}-(ii) allows to write the division rate as $\beta(\mathbf{x}) = g_1(\mathbf{x}) B(\mathbf{x})$ where function $B$ should be interpreted as an ``\textit{age hazard rate}", {a generalisation of the adder division rate introduced in Example~\ref{ex:adder}. Thus, we allow ourselves to have unbounded division rates, provided that the \textit{age hazard rate} $B$ is bounded, which will allow to control nonetheless the law of ages at division}.  Biologically this has been interpreted as individuals not perceiving actual time, but rather their own biological age, upon which the division event is decided~\cite{Taheri-Araghi2015}. The parameter $a^*$ is the minimal division age. It imposes that it is not possible to divide immediately after birth. For ages bigger than $a^*$, the bounds on $B$ allow to stochastically bound the age at division between two exponential random variables of rate parameter $\beta_-$ and $\beta_+$.  Assumption~\ref{ass:all}-(iii) imposes inexact cell divisions {which always give a bounded number of individuals, but almost surely more than 1, which sets us in the supercritical case.} The two considered cases bring together a broad family of transition kernels used in similar models. In particular the assumptions concerning the mutational kernel are inspired from~\cite{Roget2017}. Importantly, the compactness is needed to prove the existence of the eigenelements of $\mathcal{Q}$ but not for the Doeblin minorisation, which holds in more general cases. In this line, Assumption~\ref{ass:all}-(iv) is key to obtain the Doeblin minorisation condition and generalises similar requirements needed in the one-dimensional case, such as Eq. (8) of~\cite{Canizo2020} for auto-similar fragmentation kernels, and Eq. (10) of~\cite{CloezGabriel20}, or Assumption (A4) of~\cite{Roget2017} for general non-local mutation-type kernels. {Finally, a major difference with respect to classical size-structured models, is that we do not require conservation of mass during reproduction events}.

\section{Existence of the eigenelements of \texorpdfstring{$\mathcal{Q}$}{Q}}
\label{sec:eigenelementsQ}	
Now, in order to bring ourselves to the conservative setting, we begin by showing the existence of some pair of eigenelements for $\mathcal{Q}$. 	
\begin{proposition}[Existence of eigenelements]  	\label{prop:eigenelements}
Under Assumptions~\ref{ass:all}, there exist a positive constant $\lambda > 0$ and a positive function $h \in W^{1,\infty}_{\textrm{loc}}(\cx)$ such that
\[
\mathcal{Q}h = \lambda h. 
\]
\end{proposition}
To do so, we can reformulate the eigenvalue problem as a one-dimensional fixed point problem. This is a classical strategy and other applications in two-dimensional spaces can be found for example in~\cite{Doumic2007,Heijmans86,gabriel:hal-01742140}. In particular, we follow closely the arguments of~\cite{Doumic2007} which corresponds to the case $g_1 \equiv 1$ with a fragmentation kernel and with additional confinement assumption in the drift term which would allow us to work in a compact interval in one of the two dimensions. We generalise this approach here.
\begin{lemma}[Reformulation as a renewal equation]
Any pair $(\lambda,h)$ such that $\lambda > 0$ and $h \in W^{1, \infty}_{\textrm{loc}}(\cx)$ is solution almost everywhere to $\mathcal{Q} h = \lambda h $ and verifies
\begin{equation}
\lim_{t \to + \infty } h(\varphi^t(\bx)) \exp \pars{ -\int_0^t \beta \pars{\varphi^s(\bx)} ds - \lambda t }  = 0
\label{eq:limitcondition}
\end{equation}
if and only if it verifies the renewal formula
\begin{equation}
h(\bx) =\int_0^\infty h(0,z) K_{\lambda}(\bx, z) dz,
\label{eq:renewal}
\end{equation}
where
\begin{equation}
K_\lambda(\bx,z) = C_\bx \int_0^\infty e^{-\lambda t} \ p_\bx(t,z) dt
\label{eq:expLambdaT}
\end{equation}
\label{lemma:renewal}
\end{lemma}	
\begin{remark}Using Definition~\ref{def:Px} we can then write Eq. \eqref{eq:renewal} as
\begin{equation}
h(\bx) = C_\bx\  \mathbb{E}_\bx [h(0,Z) e^{-\lambda T} ] .
\label{eq:renewalEsper}
\end{equation}
\end{remark}
\begin{proof}[Proof of Lemma~\ref{lemma:renewal}. ]  We proceed by the method of integration along characteristics. { First of all, take $h \in W^{1,\infty}_{\textrm{loc}}(\cx)$ and fix some $\bx \in \cx$ and $\lambda \geq 0$. We study $\mathcal{R}^{h,\lambda}_{\bx} : \rr_+ \to \rr$ defined by
\begin{equation*}
\mathcal{R}^{h,\lambda}_{\bx}(t) := h(\varphi^t(\bx)) \exp \pars{ -\int_0^t \beta \pars{\varphi^s(\bx)} ds - \lambda t } , \quad t \geq 0.
\end{equation*}
It is clear that $\mathcal{R}^{h,\lambda}_{\bx}$ is in $L^1_{\textrm{loc}}(\rr_+)$. We show now that it is weakly differentiable. At least formally, we have that
}
\begin{align}
\frac{\partial}{\partial t} & \mathcal{R}^{h,\lambda}_{\bx}(t) \nonumber \\
=& \pars{ \nabla h(\varphi^t(\bx)) ^\top  g(\varphi^t(\bx)) - (\beta(\varphi^t(\bx)) + \lambda )  h(\varphi^t(\bx))  }  \exp \pars{ -\int_0^t \beta \pars{\varphi^s(\bx)} ds - \lambda t} ,
\label{eq:diffR}
\end{align}
{
which is well defined and in $L^1_{\textrm{loc}}(\rr_+)$ since $h \in W^{1,\infty}_{\textrm{loc}}(\cx)$, and $g$ and $\beta$ are also locally bounded from Assumptions~\ref{ass:all}. Therefore $\mathcal{R}^{h,\lambda}_{\bx} \in W^{1,\infty}_{\textrm{loc}}(\rr_+)$. Now, using the definition of $\mathcal{Q}$ we get whenever $h \in D(\mathcal{Q})$,
}
\begin{align*}
\frac{\partial}{\partial t} & \mathcal{R}^{h,\lambda}_{\bx}(t) \nonumber \\
=& \left( \mathcal{Q} h(\varphi^t(\bx)) - \lambda h(\varphi^t(\bx)) - \beta(\varphi^t(\bx)) \int_0^\infty h(0,z) k(\varphi^t(\bx),z) dz   \right) e^{ -\int_0^t \beta \pars{\varphi^s(\bx)} ds - \lambda t}   .
\end{align*}

$ \underline{\bullet \  \mathcal{Q} h = \lambda h \cap \eqref{eq:limitcondition} \implies \eqref{eq:renewal} : }$
Now, suppose that $(\lambda,h)$ is solution a.e. to $\mathcal{Q} h = \lambda h $. { Then, for almost every $t$
\begin{align}
\frac{\partial}{\partial t}  \mathcal{R}^{h,\lambda}_{\bx}(t) &= - \beta(\varphi^t(\bx))  e^{ -\int_0^t \beta \pars{\varphi^s(\bx)} ds - \lambda t}   \int_0^\infty h(0,z) k(\varphi^t(\bx),z) dz    \nonumber \\
&= - C_{\bx} e^{-\lambda t} \int_0^\infty h(0,z) p_{\bx}(t,z) dz,
\label{eq:dthExp}
\end{align}
which is well defined and integrable over $(0,+\infty)$ by Assumptions~\ref{ass:all}, and since the eigenfunction $h$ must be in the domain of the extended generator, so the integral term is well defined.
} Now, suppose that Eq. \eqref{eq:limitcondition} is also verified. Then, integrating Eq. \eqref{eq:dthExp} in $(0,+\infty)$ and using the decay condition Eq. \eqref{eq:limitcondition} results into
\[
h(\bx) = \int_0^{+\infty} C_{\bx} e^{-\lambda t} \int_0^\infty h(0,z) p_{\bx}(t,z) dz \ dt 
\]
which, by Fubini, gives exactly Eq. \eqref{eq:renewal}.

$ \underline{\bullet \ \eqref{eq:renewal} \implies \mathcal{Q} h = \lambda h \cap \eqref{eq:limitcondition} : }$ Finally, suppose that we have Eq. \eqref{eq:renewal}. Then we have:
\begin{align*}
h & (\varphi^t(\bx)) =  \int_0^\infty h(0,z) K_{\lambda}(\varphi^t(\bx), z) dz \\
=& \int_0^{\infty} \int_0^\infty h(0,z)  k(\varphi^{t+s}(\bx),z)  \beta(\varphi^{t+s}(\bx))    \exp \pars{ -\int_t^{t+s} \beta \pars{\varphi^u(\bx)} du - \lambda s}  ds dz \\
=& \left( \int_0^{\infty} \int_0^\infty h(0,z)  k(\varphi^{s}(\bx),z)  \beta(\varphi^{s}(\bx))    \exp \pars{ -\int_0^{s} \beta \pars{\varphi^u(\bx)} du - \lambda s}  ds dz \right. \\
& - \left. \int_0^{\infty} \int_0^t h(0,z)  k(\varphi^{s}(\bx),z)  \beta(\varphi^{s}(\bx))    \exp \pars{ -\int_0^{s} \beta \pars{\varphi^u(\bx)} du - \lambda s}  ds dz \right)   \\
& \times \exp \pars{ \int_0^{t} \beta \pars{\varphi^u(\bx)} du + \lambda t} 
\end{align*}
Therefore, using Eq. \eqref{eq:renewal} { again to replace the double integrals of the RHS we obtain}:
\begin{equation}
\begin{split}
{ \mathcal{R}^{h,\lambda}_{\bx}(t) = } \ h(\varphi^t(\bx)) & \exp \pars{ -\int_0^t \beta \pars{\varphi^s(\bx)} ds - \lambda t }  \\ & = h(\bx) -  \int_0^{\infty} \int_0^t h(0,z)  k(\varphi^{s}(\bx),z)  \psi(s\vert \bx) e^{-\lambda s}  ds dz \label{eq:hflux}
\end{split}
\end{equation}
As $t \to + \infty$, the improper integral in the RHS of Eq. \eqref{eq:hflux} converges towards
\[
\lim_{t \to +\infty } \int_0^{\infty} \int_0^t h(0,z)  k(\varphi^{s}(\bx),z)  \psi(s\vert \bx) e^{-\lambda s}  ds dz = \int_0^\infty h(0,z) K_\lambda (\bx, z) dz = h(\bx),
\]
from which we obtain Eq. \eqref{eq:limitcondition}. Moreover, supposing that $h \in W^{1,\infty}_{\textrm{loc}}(\cx)$, { from the previous analysis, we have that $\mathcal{R}^{h,\lambda}_{\bx} \in W^{1,\infty}_{\textrm{loc}}(\rr_+)$}, so by differentiation of Eq. \eqref{eq:hflux} we obtain almost everywhere,
\begin{align*}
 \frac{\partial}{\partial t} \mathcal{R}^{h,\lambda}_{\bx}(t) =  -  \int_0^{\infty}  h(0,z)  k(\varphi^{t}(\bx),z)  \psi(t\vert \bx) e^{-\lambda t}  dz
\end{align*}
Hence, a comparison with Eq. \eqref{eq:diffR} gives that $h \in D(\mathcal{Q})$, and for all $\bx \in \cx$ and $t > 0$ we have almost everywhere $ \mathcal{Q} h(\varphi^t(\bx)) - \lambda h(\varphi^t(\bx)) = 0$, or equivalently, $Q h = \lambda h $ almost everywhere in $\cx$.
\end{proof}
\begin{remark}
In particular the function $\eta(y) := h(0,y)$ defined for all $y\geq 0$ is solution to the fixed point problem
\begin{equation}
\eta(y) = \int_0^\infty \eta(z) K_\lambda(0,y,z) dz .
\label{eq:renewal0}
\end{equation}
\end{remark}
Therefore we will consider the operator $G_{\lambda}$ defined for $f \in C^1(\rr_+)$ by
\begin{equation}
\mathcal{G}_{\lambda} f(y) = \int_0^\infty f(z) K_\lambda(0,y,z) dz \quad \forall y > 0.
\label{eq:Glambda}
\end{equation}
We also introduce the operator $\mathcal{J}_\lambda : \mathcal{M}(\rr_+) \to \mathcal{M}(\rr_+)$ which for any Radon measure $\nu$ supported in $\rr_+$ gives
\begin{equation}
\mathcal{J}_{\lambda} \nu = \pars{\int_0^\infty K_\lambda(0,z,y) \nu(dz) } dy
\end{equation}
and verifies the duality property below:
\begin{proposition} For every $\lambda \geq 0$, $\mathcal{J}_\lambda$ is the adjoint operator of $\mathcal{G}_\lambda$.
\end{proposition}
\begin{proof}
Let $f \in C(\rr_+^2)$ and $\nu \in \mathcal{M}(\rr_+^2)$. By Fubini's Theorem,
\begin{align*}
\ap{\nu}{\mathcal{G}_\lambda f} &= \int_0^\infty \pars{\int_0^\infty f(z) K_\lambda(0,y,z) dz } \nu(dy) \\
&=  \int_0^\infty f(z) \pars{\int_0^\infty K_\lambda(0,y,z) \nu(dy) } dz = \ap{\mathcal{J}_\lambda \nu}{f}.
\end{align*}
\end{proof}
\begin{remark} From Eq. \eqref{eq:renewalEsper}, we can write
\[
\mathcal{G}_{\lambda} f(y) = C_{(0,y)} \ \mathbb{E}_{(0,y)} [f(Z) e^{-\lambda T} ] .
\]
where again $C_{(0,y)} =  \norm{K_0(0,y,\cdot)}_{1}$ is the mean number of offspring produced by an individual of initial size $y$ after its first jump.
\end{remark}	
\begin{proof}[Proof of Proposition~\ref{prop:eigenelements}]
We aim to prove that there is a unique $\lambda > 0$ for which the operator $\mathcal{G}_{\lambda}$ admits a unique fixed point $h(0,\cdot)$. The pair $(\lambda, h)$ is then solution to the eigenproblem $\mathcal{Q} h = \lambda h $. This will be proven by means of Krein-Rutman's theorem. In order to be able to apply this theorem we need to work with a strictly positive compact operator. For the compactly supported mutational kernel it is immediately the case, however it is not the case for $G_\lambda$ with a fragmentation kernel. Thus, we shall follow a standard approximation scheme for the proof which is structured as follows:
\begin{enumerate}
\item We define a truncated version of $\mathcal{G}_{\lambda}$ which by Arzéla-Ascoli's theorem we prove to be a positive compact operator in the Banach space of continuous functions.
\item We apply Krein-Rutman theorem to prove that for each $\lambda \geq 0$ the truncated operator admits a unique eigenvalue $\mu_\lambda \geq 0$ and suitably normalised eigenfunction $h_\lambda \geq 0$. 
\item We prove that there exists a unique $\lambda_0 > 0$ such that $\mu_{\lambda_0} = 1$
\item We prove that the value of $\lambda_0$ is uniformly bounded for all the members of the family of truncated operators.
\item We pass to the limit and show that the limit eigenelements $(\lambda_0, h_{\lambda_0})$ of the family of truncated operators are indeed solution to the fixed point problem.
\end{enumerate}
Note that the proof is also valid for the compact mutational kernel which verifies Assumption~\ref{ass:all}-(iii)-(b), but in that case neither the truncation nor the uniform estimates are needed. 		
\begin{enumerate}[wide, label=\textbf{Step\,\arabic* :},labelindent=0pt]
\item \textbf{Construction of the truncated operator.} 
\newline
For each $R>0$ let 
$\mathcal{G}_{\lambda}^{R} : C^1([0,R]) \to C^1([0,R])$ defined for all $\lambda > 0$, for $f \in C^1([0,R])$ by
\begin{equation}
\mathcal{G}_{\lambda}^{R} f(y) = \int_0^R f(z) K_\lambda^{R}(0,y, z) dz  \quad \forall y \in (0,R)
\label{eq:def_truncation}
\end{equation}
with
\[
K_\lambda^{R}(0,y, z)  =  \int_0^\infty \pars{ k (\varphi^t (0,y),z ) + \frac{1}{R} \int_R^{\infty} k (\varphi^t (0,y),\zeta) d\zeta  } \psi(t\vert (0,y)) e^{- \lambda t} dt
\]
We require to add the uniform correction $z \mapsto \frac{1}{R} \int_R^{\infty} k (\varphi^t (0,y),\zeta) d\zeta$ in order to endorse the strict positivity of the operator. Indeed, for all $y \in [0,R]$, from Fubini's theorem, Assumption~\ref{ass:all}-(iii) and Jensen's inequality we obtain
\begin{align*}
\int_0^R K_\lambda^{R}(0,y, z)   dz &= \int_0^\infty \pars{ \int_0^\infty k (\varphi^t (0,y),z ) dz  } \psi(t\vert (0,y)) e^{- \lambda t} dt \\
&> \int_0^\infty e^{-\lambda t} \psi(t \vert  (0,y)) dt \\
&\geq \exp \pars{-\lambda \esper{(0,y)}{T}} .
\end{align*}
Moreover, Assumption 3.3-(ii) gives that
\[
0 < \espero{ \phi^{-1}_{(0,y)} \pars{{\mathbf{A}}_- , Y_{(0,y)}({\mathbf{A}}_- ) } }  	\leq \esper{(0,y)}{T} \leq \espero{ \phi^{-1}_{(0,y)} \pars{{\mathbf{A}}_+ , Y_{(0,y)}({\mathbf{A}}_+ ) } } < +\infty
\]
where ${\mathbf{A}}_-$ (respectively ${\mathbf{A}}_+$) follows an Exponential distribution of parameter $\beta_-$ (respectively $\beta_+$). Therefore for all positive $f \in C^1([0,R])$, $G_\lambda^R f > 0$.
\newline
Moreover, if in analogy with Definition~\ref{def:Px}, we define for all $R>0$ the random couple $(T_R,Z_R) \in \rr_+ \times [0,R]$ such that under $\mathbb{P}_{(0,y)}$ they have joint probability density
\[
p_{(0,y)}^{R}(t,z) = \frac{1}{ C_{(0,y)}} \pars{ k (\varphi^t (0,y),z ) + \frac{\mathds{1}_{z \leq R}}{R} \int_R^{\infty} k (\varphi^t (0,y),\zeta) d\zeta    } \psi(t\vert  (0,y)) ,
\] 
then we can write
\begin{equation}
\mathcal{G}_{\lambda}^{R} f(y) = C_{(0,y)}  \  \mathbb{E}_{(0,y)} [f(Z_R) e^{-\lambda T_R} \mathds{1}_{Z_R \leq R}].
\label{eq:G_esper}
\end{equation}
\newline
\item \textbf{Existence of the eigenelements of $\mathcal{G}_{\lambda}^{R}$.}
\newline
We begin by proving that for all $\varepsilon > 0, \lambda \geq 0$ and $R>0$, $\mathcal{G}_{\lambda}^{R}$ is compact. We show that for every sequence $(f_n)_n$ in the unit ball of $C([0,R])$ there exists a subsequence of $\pars{\mathcal{G}_{\lambda}^{R} f_n}_n$ which converges in $C([0,R])$ equipped with the uniform norm.
\begin{enumerate}[i.]
\item \textbf{Uniform bound: }	For all $y \in (0,R)$, $f$ in the unit ball of $C[0,R]$ we have from Eq. \eqref{eq:G_esper}:
\begin{align*}
\mathcal{G}_{\lambda}^{R} f (y)  &\leq C_{(0,y)} \norm{f}_{\infty} \leq  \bar K
\end{align*}
\item \textbf{Equicontinuity: } Since $g_1 \in C^1(\rr_+^2)$ and is strictly positive, and $k$ is continuous in the first two variables, we have that for every $\lambda \geq 0$, $(y,z) \in [0,R] \times [0,R] \mapsto K_{\lambda}(0,y,z)$ is an uniformly continuous function on $[0,R] \times [0,R]$. Therefore for all $\lambda \geq 0$ and $\varepsilon > 0$ there exists $\delta > 0$ such that if $\vert y_1 - y_2\vert  + \vert z_1 - z_2\vert  < \delta$ for $y_1,y_2,z_1,z_2 \in [0,R]$, then $\abs{K_\lambda(0,y_1,z_1) - K_\lambda(0,y_2,z_2) } < \varepsilon/R$.  Hence, for all $f$ in the unit ball, $y_1, y_2 \in [0,R]$ such that $\vert y_1 - y_2\vert  < \delta$ we have:
\begin{align*}
\abs{ \mathcal{G}_{\lambda}^{R} f (y_1) - \mathcal{G}_{\lambda}^{R} f (y_2) }
&\leq \int_0^R  \vert f(z) \vert  \vert  K_{\lambda}(0,y_1,z) - K_{\lambda}(0,y_2,z)\vert   dz   < \varepsilon
\end{align*}
independently on $y_1, y_2$.
\end{enumerate}
Finally, by Ascoli's criterium, there exists a convergent subsequence of $\pars{\mathcal{G}_{\lambda}^{R} f_n}_n$ and so the operator $\mathcal{G}_{\lambda}^{R}$ is strictly positive and compact for the uniform topology of $C([0,R])$.		
Therefore, by Krein-Rutman theorem~\cite{Perthame2007} there exists a unique triplet of a positive real value $\mu_{\lambda}^{R} > 0$, function $\eta_{\lambda}^{R} > 0$ continuous on $[0,R]$, and a positive Radon measure $\nu_{\lambda}^{R}$ supported on $[0,R]$ such that 
\begin{align}
\mathcal{G}_{\lambda}^{R} \eta_{\lambda}^{R}  &= \mu_{\lambda}^{R} \eta_{\lambda}^{R} 
\label{eq:KR_h}\\
\mathcal{J}_{\lambda}^{R} \nu_{\lambda}^{R}  &= \mu_{\lambda}^{R} \nu_{\lambda}^{R}
\label{eq:KR_nu} \ , \ \nu_{\lambda}^{R}([0,R]) = 1 \\
\ap{\nu_{\lambda}^{R}}{\eta_{\lambda}^{R}}_R &= 1, 
\label{eq:KR_normalisation}
\end{align}
where we denote $
\ap{\nu}{f}_R = \int_0^R f(y) \nu(dy) $.
\newline
\item \textbf{Existence and uniqueness of $\lambda_0 > 0$ such that $\mu_{\lambda_0}^{R} = 1$}
\newline
We show that the mapping $\lambda \mapsto \mu_{\lambda}^{R}$ is a continuous strictly decreasing function which goes through the value of 1 at some point.
First, note that from Equations \eqref{eq:KR_h} and \eqref{eq:KR_normalisation}, we have
\begin{align}
\ap{\nu_{\lambda}^{R}}{\mathcal{G}_{\lambda}^{R} \eta_{\lambda}^{R}}_R = \mu_{\lambda}^{R}
\label{eq:normCondition}
\end{align}
We prove that $\lambda \mapsto \ap{\nu_{\lambda}^{R}}{\mathcal{G}_{\lambda}^{R} \eta_{\lambda}^{R}}_R$ is differentiable continuous and decreasing. Let us consider the derivatives in the sense of distributions $\partial_\lambda \nu_\lambda^{R}$ and $\partial_\lambda \eta_\lambda^R$. We show below that $\lambda \mapsto \ap{\nu_{\lambda}^{R}}{\mathcal{G}_{\lambda}^{R} \eta_{\lambda}^{R}}_R$ is actually strongly differentiable with respect to $\lambda$ as it has the same regularity as $\lambda \mapsto \mathcal{G}_\lambda^{R} f$. First, by dominated convergence, differentiating under the integral sign on Eq. \eqref{eq:G_esper} gives for every $f \in C^1([0,R])$,
\begin{equation}
\pars{ \partial_\lambda \mathcal{G}_{\lambda}^{R}} f(y) = - C_{(0,y)}  \  \mathbb{E}_{(0,y)} [f(Z_R) T_R e^{-\lambda T_R} \mathds{1}_{Z_R \leq R}].
\label{eq:dG_esper}
\end{equation}
Then, by differentiating under the duality brackets, and using the duality between $\mathcal{G}$ and $\mathcal{J}$ with \eqref{eq:KR_h} and \eqref{eq:KR_nu}, we obtain
\begin{align*}
\partial_\lambda 	\mu_{\lambda}^{R} &= \ap{\partial_\lambda  \nu_{\lambda}^{R}}{\mathcal{G}_{\lambda}^{R} \eta_{\lambda}^{R}}   + \ap{\nu_{\lambda}^{R}}{\mathcal{G}_{\lambda}^{R} \pars{ \partial_\lambda  \eta_{\lambda}^{R}}} + \ap{\nu_{\lambda}^{R}}{\pars{ \partial_\lambda  \mathcal{G}_{\lambda}^{R} } \eta_{\lambda}^{R}}  \\
&= \ap{\partial_\lambda  \nu_{\lambda}^{R}}{\mathcal{G}_{\lambda}^{R} \eta_{\lambda}^{R}}  + \ap{\mathcal{J}_{\lambda}^{R}  \nu_{\lambda}^{R}}{ \partial_\lambda  \eta_{\lambda}^{R}} + \ap{\nu_{\lambda}^{R}}{\pars{ \partial_\lambda  \mathcal{G}_{\lambda}^{R} } \eta_{\lambda}^{R}}  \\
&= \mu_{\lambda}^{R} \pars {\ap{\partial_\lambda  \nu_{\lambda}^{R}}{ \eta_{\lambda}^{R}}  + \ap{ \nu_{\lambda}^{R}}{ \partial_\lambda  \eta_{\lambda}^{R}}} + \ap{\nu_{\lambda}^{R}}{\pars{ \partial_\lambda  \mathcal{G}_{\lambda}^{R} } \eta_{\lambda}^{R}} \\
&= \mu_{\lambda}^{R} \partial_\lambda \ap{ \nu_{\lambda}^{R}}{ \eta_{\lambda}^{R}} + \ap{\nu_{\lambda}^{R}}{\pars{ \partial_\lambda  \mathcal{G}_{\lambda}^{R} } \eta_{\lambda}^{R}}
\end{align*}
Eq. \eqref{eq:KR_normalisation} gives $\partial_\lambda \ap{ \nu_{\lambda}^{R}}{ \eta_{\lambda}^{R}}  = 0$, and therefore $\partial_\lambda 	\mu_{\lambda}^{R} = \ap{\nu_{\lambda}^{R}}{\pars{ \partial_\lambda  \mathcal{G}_{\lambda}^{R} } \eta_{\lambda}^{R}}$, i.e.,
\begin{align}
\partial_{\lambda} \mu_{\lambda }^{R} = -  \int_0^R C_{(0,y)} \ \mathbb{E}_{(0,y)} \sqkets{\eta_{\lambda}^{R} (Z_R) T_R e^{-\lambda T_R} \mathds{1}_{Z_R \leq R} } \nu_{\lambda}^{R}(dy)
\end{align}
Since all the integrands are non-negative we have $\partial_\lambda 	\mu_{\lambda}^{R} < 0$. So $\lambda \mapsto \mu_{\lambda}^{R}$ is a continuous strictly-decreasing function. 
Moreover, doing $\lambda = 0$, integrating Eq. \eqref{eq:KR_nu}, using Fubini's theorem to integrate first in the $z$ variable, and using Assumption~\ref{ass:all}-(iii), we obtain
\begin{align*}
\mu_0^{R} &= \int_0^R \mathcal{J}_{0}^{R} \nu_0^{R}(dz)  \\
&= \int_0^R \int_0^R \int_0^{\infty} \pars{ k (\varphi^t (0,y),z ) + \frac{\int_R^{\infty} k (\varphi^t (0,y),\zeta) d\zeta}{R}  } \psi(t\vert (0,y)) dt \  \nu_{0}^{R}(dy) \ dz \\
&=  \int_0^R \int_0^{\infty} \pars{ \int_0^\infty k (\varphi^t (0,y),z )  dz } \psi(t\vert (0,y)) dt \  \nu_{0}^{R}(dy) > 1
\end{align*}
On the other hand, doing $\lambda \to \infty$, passing to the limit under the expecation of Eq. \eqref{eq:G_esper} we get for every $f \in C([0,R])$, $\mathcal{G}_{\lambda}^{R}f  \to 0$ uniformly as $\lambda \to \infty$. In particular, by the equicontinuity of $\mathcal{G}_{\lambda}^{R}$, for every $\delta \in (0,2)$, there must be $\lambda_*$ large enough such that for every $f \in C([0,R])$, $\mathcal{G}_{\lambda}^{R}f  \leq \delta$ for all $\lambda \geq \lambda_*$ and hereby, $\mu_{\lambda}^{R} \leq \delta$ for all $\lambda \geq \lambda_*$. Therefore $\mu_{\lambda}^{R} \to 0$ as $\lambda \to \infty$.
In consequence, there must be a unique $\lambda_0 > 0$ such that $\mu_{\lambda_0}^{R} = 1$. 
We then define $\lambda_{R}$ as the only $\lambda_0 > 0$ such that $\mu_{\lambda_0}^{R} = 1$ and denote $\eta_{R} = \eta^{R}_{\lambda_{R}}$ the respective eigenfunction. Next, we construct a sequence of $h_{R}$ from $\eta_{R}$ which are to converge to the solution of the intial eigenproblem and we show that we can establish an uniform bound over $\lambda_{R}$. 
\newline
\item \textbf{Construction of $h_{R}$.}
\newline
We extend the definition of $K_\lambda$ to all $(a,y) \in \cx, z \in [0,R]$. Define
\[
K_\lambda^{R}(a,y,z)  := \int_0^\infty \pars{ k (\varphi^t (a,y),z ) + \frac{\int_R^\infty k(\varphi^t(0,y),\zeta) d \zeta }{R} } \psi(t\vert (a,y)) e^{- \lambda t} dt,
\] 
and let 
\begin{equation}
h_{R}(a,y) := \int_0^R \eta_{R}(z) K_{\lambda_{R}}^{R}(a,y,z) dz.
\label{eq:h_epsR}
\end{equation}
Hence, taking $a=0$, since $\eta_{R}$ solves Eq. \eqref{eq:KR_h} for $\mu_\lambda^{R} = 1$, we have that:
\[
h_{R}(0,y) = \int_0^R \eta_{R}(z) K_{\lambda_{R}}^{R}(0,y,z) dz = \mathcal{G}_{\lambda_{R}}^{R} \eta_{R}(y) = \eta_{R}(y),
\] 
and therefore $h_R$ verifies
\begin{equation}
\begin{cases}
h_{R}(\bx) = \int_0^R h_{R}(0,z) K_{\lambda_{R}}^{R}(\bx,z) dz = C_{\bx}  \ \mathbb{E}_{\bx} \sqkets{\eta_{R}(Z_R) e^{-\lambda_{R} T_R} \mathds{1}_{Z_R \leq R} } \quad \forall \bx \in \cx \\
h_{R}(0,y) =\eta_{R}(y) \quad \forall y \in (0,R)
\end{cases}
\label{eq:renewal_truncated}
\end{equation}
where $C_\bx = \norm{K_0(\bx, \cdot)}_{L_1(\rr_+)}$.
Then, we can repeat the steps of the proof of Lemma~\ref{lemma:renewal} to show that the truncated renewal equation \eqref{eq:renewal_truncated} (which is the truncated version of Eq. \eqref{eq:renewal}) is equivalent to have the boundary condition
\begin{equation}
\lim_{t \to + \infty } h_R(\varphi^t(\bx)) \exp\pars{ - \int_0^t \beta(\varphi^s(\bx)) ds - \lambda_R t } = 0
\label{eq:truncatedBoundary}
\end{equation}
and to have that $h_{R}$ is solution to the truncated eigenvalue problem
\[
\mathcal{Q}_{R} h_{R}(a,y) = \lambda_{R} \  h_{R}(a,y) 
\]
where
\begin{align*}
\mathcal{Q}_{R} h(a,y) &= g(a,y)^\top \nabla h(a,y) \\
& \ + \beta (a,y)  \pars{ \int_0^R    h(0,z)  \pars{ k(a,y,z) + \frac{\int_R^\infty k(a,y,\zeta) d\zeta }{R}}  dz -  h(a,y)}. 
\end{align*}
Hence, developing $\mathcal{Q}_{R} h_{R}(0,y)$ one obtains 
\begin{align*}
\mathcal{Q}_{R} h_{R}(0,y) =& g_1(0,y) \partial_a h_{R}(0,y) + g_2(0,y) \partial_y h_{R}(0,y) \\
& + \beta(0,y)  \pars{ \int_0^R    h_{R}(0,z)  \pars{ k(a,y,z) + \frac{\int_R^\infty k(a,y,\zeta) d\zeta }{R} }  dz -  h_{R}(a,y)}.
\end{align*}
Therefore $\eta_{R} = h_{R}(0,\cdot)$ is solution to
\begin{align}
\lambda_{R} & \eta_{R}(y)  = \ g_2(0,y) \eta_{R}'(y) 
- \beta(0,y)  \eta_{R}(y)  \nonumber \\
&+ \beta (0,y)   \int_0^R    \eta_{R}(z)  \pars{ k(0,y,z) + \frac{\int_R^\infty k(a,y,\zeta) d\zeta }{R} + g_1(0,y) \partial_a K_{\lambda_{R}}^{R}(0,y,z) }  dz 
\label{eq:transportEq}
\end{align}
In our case, Assumption~\ref{ass:all}-(ii) which imposes $\beta(0,y) = 0$ for every initial size $y$ simplifies this last equation into
\[
g_2(0,y) \eta_{R}'(y) 	- \lambda_{R} \ \eta_{R}(y) = 0
\] 
Therefore for all $R > 1$, if we impose the normalisation condition $\eta_{R}(1) = 1$, we have
\begin{equation}
\eta_{R}(y) = \exp \pars{ \lambda_{R} \int_1^y \frac{1}{g_2(0,z)} dz }  , \ y \in [0,R]
\label{eq:h_charac}
\end{equation}
Finally, coming back to \eqref{eq:G_esper} and \eqref{eq:KR_h}, we have for all $y \in (0,R)$,
\begin{align*}
\eta_{R}(y) &= C_{(0,y)}  \  \mathbb{E}_{(0,y)} [\eta_R(Z_R) e^{-\lambda T_R} \mathds{1}_{Z_R \leq R}] \\
\iff 1 &= C_{(0,y)}  \  \mathbb{E}_{(0,y)} \sqkets{\frac{\eta_R(Z_R)}{\eta_R(y)} e^{-\lambda T_R} \mathds{1}_{Z_R \leq R}} \\
\iff 1 &= C_{(0,y)}  \  \mathbb{E}_{(0,y)}  \sqkets{\exp \pars{ \lambda_R \pars {\int_y^{Z_R} \frac{1}{g_2(0,z)} dz - T_R } }}
\end{align*}
In particular the last equation characterises $\lambda_{R}$ as the unique $\lambda > 0$ such that for all $y \in (0,R)$, the following Euler-Lotka-type equation is verified
\begin{equation}
1 = C_{(0,y)}   \mathbb{E}_{(0,y)}\sqkets{\exp \pars{ \lambda \pars {\int_y^{Z_R} \frac{1}{g_2(0,z)} dz - T_R } }}.
\label{eq:EulerLotka}
\end{equation}
\newline
\item \textbf{Uniform bound for $\lambda_{R}$ } (Fragmentation case)
\newline
Suppose that for all $a\geq 0$, $\textrm{supp } k(a,y,\cdot) \subseteq (0,y)$. This is, the newborns sizes are almost surely smaller than the parent size. Hence, for all initial size $y$ we have
\begin{equation}
\mathbb{P}_{(0,y)} \pars{ T_R > \int_y^{Z_R} \frac{1}{g_2(A_{0,y}(z),z)} dz } = 1.
\end{equation}
Indeed, from Lemma~\ref{lemma:flowprop}-(4.) we have that $\phi^{-1}_{0,y}(A_{0,y}(z),z) = \int_y^{z} \frac{1}{g_2(A_{0,y}(z),z)} dz$ is the time needed to go from size $y$ to $z$ following the deterministic flow only, and it has to be smaller than the division time at which the trajectory jumps to $z$. Then, thanks to Assumption~\ref{ass:all}-(i) which gives $g_2(0,y) \geq g_2(a,y)$, we have also that
\[
\mathbb{P}_{(0,y)} \pars{ T_R > \int_y^{Z_R} \frac{1}{g_2(0,z)} dz } = 1.
\] 
Therefore for all $\lambda > 0$
\begin{equation}
\exp \pars{ \lambda \pars {\int_y^{Z_R} \frac{1}{g_2(0,z)} dz - T_R } } \leq 1 \in L^1(\rr_+^2, p_{(0,y)} dt dz ),\quad  \mathbb{P}_{(0,y)}\textrm{-a.s.},
\label{eq:Step5_asBound}
\end{equation}
and by dominated convergence if $\lambda_{R}$ converges to $+\infty$ as $R \to \infty$, then 
\[
\mathbb{E}_{(0,y)} \sqkets{\exp \pars{ \lambda_{R} \pars {\int_y^{Z_R} \frac{1}{g_2(0,z)} dz - T_R } }} \to 0 
\]
which contradicts Eq. \eqref{eq:EulerLotka}.  So there must exist $\bar \Lambda > 0$ such that for all $R > 1$, $\lambda_{R} < \bar \Lambda $. 
Moreover, analogous to Step~3, if we differentiate Eq. \eqref{eq:normCondition} in the sense of distributions with respect to $R$, we obtain 
\begin{align*}
\partial_R 	\mu_{\lambda}^{R} &=  \ap{\nu_{\lambda}^{R}}{\pars{ \partial_R  \mathcal{G}_{\lambda}^{R} } \eta_{\lambda}^{R}}.
\end{align*}
Again, the definition $\mathcal{G}_{\lambda}^{R}$ gives us that this derivative can be computed in the strong sense. Indeed, for any positive continuous function $f: [0,R] \to \rr_+$ we have
\begin{align*}
\partial_R & \mathcal{G}_{\lambda}^{R} f(y) \\
=& \frac{\partial}{\partial R} \int_0^R f(z) \int_0^\infty \pars{ k (\varphi^t (0,y),z ) + \frac{\int_R^\infty k (\varphi^t (0,y),\zeta ) d\zeta  }{R} } \psi(t\vert (0,y)) e^{- \lambda t} dt dz \\
=& f(R) \pars{ \int_0^\infty \pars{ k (\varphi^t (0,y),R ) + \frac{\int_R^\infty k (\varphi^t (0,y),\zeta ) d\zeta  }{R} } \psi(t\vert (0,y)) e^{- \lambda t} dt }   \\
&- \int_0^R f(z) \int_0^\infty \frac{1}{R} \pars{k (\varphi^t (0,y),R ) + \frac{\int_R^\infty k (\varphi^t (0,y),\zeta ) d\zeta  }{R}} \psi(t\vert (0,y)) e^{- \lambda t } dt dz \\
=&  \pars{\int_0^R \frac{f(R)-f(z)}{R} dz } \\
&\times \pars{ \int_0^\infty \pars{ k (\varphi^t (0,y),R ) + \frac{\int_R^\infty k (\varphi^t (0,y),\zeta ) d\zeta  }{R} } \psi(t\vert (0,y)) e^{- \lambda t} dt },
\end{align*}
which is positive whenever $f$ is an increasing function. Since Eq. \eqref{eq:h_charac} gives that for every fixed $\lambda$, $\eta_\lambda^R(y)$ is increasing in $y$, then $\pars{ \partial_R  \mathcal{G}_{\lambda}^{R} } \eta_{\lambda}^{R} > 0$ and therefore $\partial_R \mu_{\lambda}^{R} > 0$. In particular, the sequence of $\lambda_{R}$, which is defined as the values of $\lambda$ such that $\mu_{\lambda}^{R} = 1$, is then also increasing in $R$. 
\newline
\item \textbf{Identification of the limit}
\newline
Step~5 gives that $\pars{\lambda_{R}}_R$ is an increasing bounded sequence as $R \to \infty$, so with a limit written $\lambda > 0$. 
Moreover, for each $\lambda_{R}$ exists a unique $h_{R}$ associated, defined by Eq. \eqref{eq:h_epsR}. The family of $h_{R}$ is equibounded and equicontinuous thanks to Eq. \eqref{eq:truncatedBoundary}, Eq. \eqref{eq:h_charac} and the bound on $\lambda_{R}$. Note indeed that Eq. \eqref{eq:h_charac} depends on $R$ only through $\lambda_{R}$. We can therefore extract a subsequence converging to some $(\lambda, h)$ as $R \to \infty$.
We must now check that $(\lambda,h)$ is a good pair of eigenelements, which we do by dominated convergence. In Step~4 we have constructed $h_{R}$ such that it is solution to Equations \eqref{eq:renewal_truncated} and \eqref{eq:truncatedBoundary} which we repeat below to justify each limit.
\[
\begin{cases}
h_{R}(\bx) = C_{\bx} \ \mathbb{E}_{\bx} \sqkets{\eta_{R}(Z_R) e^{-\lambda_{R} T_R} \mathds{1}_{Z_R \leq R} } \quad \forall \bx \in \cx \\
h_{R}(0,y) =\eta_{R}(y) \quad \forall y \in (0,R) \\
h_R(\varphi^t(\bx)) \underset{t \to \infty}{\sim} \exp \pars{\int_0^t \beta(\varphi^s(\bx)) ds + \lambda_R t}  
\end{cases}
\]
The normalisation constant $C_\bx$ is already the one required in the limit case. For the expectation term, recalling from Eq. \eqref{eq:h_charac} that
\[
\frac{\eta_R(y_2)}{\eta_R(y_1)} = \exp \pars{\lambda_R \int_{y_1}^{y_2} \frac{1}{g_2(0,z)} dz }
\]
and using Eq. \eqref{eq:Step5_asBound} in Step~5, we deduce that for all $y \in (0,R)$, 
\[
\eta_{R}(Z_R) e^{- \lambda_{R} T_R} \leq \eta_{R}(y)  \quad \mathbb{P}_{(0,y)}\textrm{-a.s.}
\]
Therefore for all $R > 1$,
\begin{align*}
\mathbb{E}_{(a,y)} \sqkets{\eta_{R}(Z_R) e^{-\lambda_{R} T_R} \mathds{1}_{Z_R \leq R} } & \leq  \eta_{R}(y) < +\infty 
\end{align*}
and we can pass to the limit under the expectations and conclude that the limit $h$ and $\lambda$ verify the renewal formula
\[
h(\bx) = C_\bx  \  \mathbb{E}_{\bx} \sqkets{h(0,Z) e^{-\lambda T} } \quad \forall \bx \in \cx 
\]
which is Eq. \eqref{eq:renewalEsper} and is equivalent to Eq. \eqref{eq:renewal}. Thus, by Lemma 4.2 the couple $(\lambda, h)$ is almost everywhere solution to $\mathcal{Q}h = \lambda h $. 
\end{enumerate}
\end{proof}
\begin{remark}
The assumption $\beta(0,\cdot) \equiv 0$ is crucial for the characterisation of $h$ in Step~4 of the proof of Proposition~\ref{prop:eigenelements}. The case $\beta(0,x) > 0$ could possibly be treated, but it would require additional assumptions in order to have $a \mapsto K_\lambda^{R}(a,x,z) \in W^{1,1}_{loc}(\rr_+)$ and to then control the age derivatives of the kernel $K_\lambda^{R}$. Then, Eq. \eqref{eq:transportEq} would be a scalar transport equation for $h_{\lambda}^{R}$, which thereby admits an elliptic maximum principle. 	Nonetheless, the assumption $\beta(0,\cdot) \equiv 0$, while being perfectly biologically meaningful, allows us to avoid this technicalities. 
\end{remark}

\section{Petiteness of compact sets for sampled chains}
\label{sec:Doeblin}
We want to prove the following Doeblin petite-set condition for all the compact sets of $\cx$.	
\begin{proposition}
	Let $P_t$ be the Markov process characterised by the infinitesimal generator $\mathcal{A}$ defined by Eq. \eqref{eq:A}. If Assumptions~\ref{ass:all} are verified, then every compact $\mathscr{K} \subset \rr_+^2$ is a petite-set for some skeleton chain of $P_t$. This is, there is a non-trivial discrete sampling measure $\mu$ over $\rr_+$ and a non-trivial measure $\nu$ over $\rr_+^2$ such that
	\[
	\ap{\mu}{\delta_\bx P_{\cdot} f} = \int_0^\infty P_t f(\mathbf{x}) \mu(dt) \geq \ap{\nu}{f} \quad \forall \bx \in \mathscr{K}
	\]
	\label{lemma:doeblinmin}
\end{proposition}
Before the proof we will introduce some useful lemmas. First, we recall Duhamel formula \eqref{eq:Duhamel}, which describes the trajectories driven by the semigroup $P_t$ and allows us to extend the definition of the semigroup as the mild solution of an iterative evolution equation. 
\begin{lemma}[Duhamel formula]
	For all $\bx \in \cx$, $f \in C_b^{1,1}(\cx)$, $P_t$ is the mild solution to
	\begin{align}
	P_t f(\mathbf{x}) =& f \pars{ \varphi^{t} (\mathbf{x}) } \Psitx{t}{\bx}  \nonumber \\
	&+ \int_0^t \psi(s\vert \bx) \int_0^\infty P_{t-s} f(0,z)  \frac{h(0,z) k \pars{\varphi^{s} (\mathbf{x}), z}}{\int_0^\infty h(0,z') k \pars{\varphi^{s} (\mathbf{x}), z'  }dz' } dz ds,
	\label{eq:Duhamel}
	\end{align}
\end{lemma}
\begin{proof}
	A classical probabilistic proof consists in writing $P_t f(\mathbf{x})$ conditionally to the occurrence of the first jump. It is also possible to prove it by means of a variation of parameters method, as in Corollary 1.7 from~\cite{Engel1999OneparameterSF}, for example. Here we provide the probabilistic proof.
	Let $X$ a Markov process whose law is given by generator $\mathcal{A}$ defined in Eq. \eqref{eq:A}. Recall from definition~\ref{def:Px} the random variables $T$ and $Z$ which represent the time of the first jump and the new size after the first jump. Note however that the transition kernel of the Markovian generator $\mathcal{A}$ has been rescaled, so that the joint law of $(T,Z)$ under $\mathbb{P}_\bx$ is from now on given by the density function
	\[
	p_{\bx}(t,z) = \psi(t\vert \bx) \cdot \frac{\beta(\bx) \frac{h(0,z)}{h(\bx)} k (\varphi^t(\bx),z)  }{\int_0^\infty \beta(\bx) \frac{h(0,z')}{h(\bx)} k (\varphi^t(\bx),z') dz' } = \psi(t\vert \bx) \cdot \frac{h(0,z) k \pars{\varphi^{s} (\mathbf{x}), z}}{\int_0^\infty h(0,z') k \pars{\varphi^{s} (\mathbf{x}), z'  }d z' }
	\]
	where the probability density of the transition $\bx \mapsto (0,z)$ is computed as the ratio between the transition rate of $\bx \mapsto (0,z)$ and the total transition rate, as described by the generator $\mathcal{A}$.
	Hence, by conditioning on $T$ under $\mathbb{P}_\bx$ and using the strong Markov property of $X$, we have:
	\begin{align*}
	P_t f(\mathbf{x}) = \esper{\bx}{f(X_t)} =& 
	\esper{\bx}{f(X_t) \mathds{1}_{T > t}} + \esper{\bx}{f(X_t) \mathds{1}_{T \leq t}} \\
	=& \esper{\bx}{f(X_t) \vert  T > t} \mathbb{P}_\bx(T > t) + \esper{\bx}{ \esper{\bx}{ \left. f \pars{X_t} \right\vert  T} \mathds{1}_{T \leq t} } \\
	=& \esper{\bx}{f(X_t) \vert  T > t} \mathbb{P}_\bx(T > t) + \esper{\bx}{  \esper{(0,Z)}{f(X_{t-T})} \mathds{1}_{T \leq t} }  \\
	=& f \pars{ \varphi^{t} (\mathbf{x}) } \Psitx{t}{\bx}  \\
	&+ \int_0^t \psi(s\vert \bx)  \int_{0}^{\infty} P_{t-s} f(0,z)  \frac{h(0,z) k \pars{\varphi^{s} (\mathbf{x}), z}}{\int_0^\infty h(0,z') k \pars{\varphi^{s} (\mathbf{x}), z'  }dz' } dz ds.
	\end{align*} 
\end{proof}
We can give now the proof of Proposition~\ref{lemma:doeblinmin}:
\begin{proof}[Proof of Proposition~\ref{lemma:doeblinmin}]
	Let $\bx \in \mathscr{K}$ compact such that $\mathscr{K} \subset [\underline a,\bar a] \times [\underline y, \bar y]$. We iterate once Duhamel's formula \eqref{eq:Duhamel}, using the positivity of $P_t$:
	\begin{align}
	P_t f(\mathbf{x}) 
	=& f \pars{ \varphi^{t} (\mathbf{x}) } \Psitx{t}{\bx}  \nonumber \\
	& + \int_0^t \psi(s\vert \bx)   \int_{0}^\infty \frac{h(0,z) k \pars{\varphi^{s} (\mathbf{x}), z}}{\int_0^\infty h(0,z') k \pars{\varphi^{s} (\mathbf{x}), z'  }dz' }  \bigg\{   \nonumber \\
	& \qquad  f \pars{ \varphi^{t-s} (0,z) } \Psitx{t-s}{(0,z)}  \nonumber \\
	&  \qquad + \int_0^{t-s} \psi(u\vert \bx)   \int_0^\infty P_{t-s-u} f(0, \xi) \frac{h(0,\xi) k \pars{\varphi^{u} (0,z), \xi}}{\int_0^\infty h(0,\xi') k \pars{\varphi^{u} (0,z), \xi'  }d\xi' } d\xi du \bigg\}  dz ds \nonumber \\
	\geq&  \int_0^t \psi(s\vert \bx)   \int_0^\infty  f \pars{ \varphi^{t-s} \pars{0 , z}  } \Psitx{t-s}{(0,z)}  \nonumber \\
	& \qquad \qquad \frac{h(0,z) k \pars{\varphi^{s} (\mathbf{x}), z}}{\int_0^\infty h(0,z') k \pars{\varphi^{s} (\mathbf{x}), z'  }dz' } dz ds 
	\label{eq:lastDuhamelIter}
	\end{align}
	To obtain the desired result we aim to solve two crucial steps: 
	\begin{enumerate}[label=\roman*.]
		\item First, to prove the existence of some $C^1$-diffeomorphism which could allow us to change variables inside the latter integral as to obtain a measure over $\cx$.
		\item Second, to bound from below the resulting integral uniformly for every $\bx \in \mathscr{K}$, using its compactness.
	\end{enumerate}
	Fix some final time $t \geq 0$, and define $\gamma_{t}: \cx  \to \cx$ as
	\[
	\gamma_{t}(s,z): = \varphi^{t-s} \pars{0, z} .
	\]
	We show first that it's a differentiable function. Fix $s,z$ and suppose $$(a,y) = \gamma_{t}(s,z).$$ Then the function $u$ defined as $u(s) = \gamma_{t}(s,z) $ is the unique solution to the Initial Value Problem
	\begin{align*}
	\begin{cases}
	u'(s) &= - g(u(s)) , \ s \leq t \\
	u(0) &= (a,y)
	\end{cases}
	\end{align*}
	Thus, $\partial_s \gamma_t(s,z) = -g(u(s))$. Moreover, by Lemma~\ref{lemma:flowprop}, the smoothness of the vector field $g$ and the fact that the ODE system is autonomous gives the smoothness of the flow with respect to the initial condition. Thus, the Jacobian matrix of $\gamma_t$ equals for all $s\leq t$ and $z>0$:
	\begin{equation}
	\mathcal{D} \gamma_{t} (s,z) = \begin{bmatrix} 
	- g\pars {\varphi^{t-s}(0,z)} &&  \partial_z \varphi^{t-s}(0,z)
	\end{bmatrix},
	\label{eq:JacobianMatrix}
	\end{equation}
	where, from Lemma~\ref{lemma:flowprop}-2, the derivative of the flow with respect to the initial size is given by 
	\[
	\partial_z \varphi^t (0,z) = \exp \pars{
		\int_0^t \mathcal{D} g \pars { \varphi^s\pars{0,z}} 	ds
	} \begin{pmatrix} 0 \\ 1 \end{pmatrix} .
	\]
	where we recall that $\mathcal{D} g $ stands for the Jacobian matrix of $g$ and $\exp(\cdot)$ is an exponential matrix.
	Moreover, let $r \mapsto Y_{(a,y)}(r)$ be the unique orbit of the vector field $g$ passing trough the point $(a, y )$. Its is straightforward that $z = Y_{(a,y)}(0)$, so that the inverse of $\gamma_t$ is given for all $(a,y) \in \rr_+^2$ by
	\[
	\gamma^{-1}_{t}(a,y) = \pars{t - \phi^{-1}_{0, Y_{(a,y)}(0)}(a,y) \ , \ Y_{(a,y)}(0)}.
	\]
	Fig.~\ref{fig:diffeo_lconstant} summarises graphically the change of variables and the definition of $\gamma^{-1}_t$. Given $a,y,\mathbf{x}$ and $t$, the inversion of $\gamma$ consists in determinating the value of ordinate $z$ when the integral curve flowing towards $(a,y)$ hits the $y$-axis and the time $t-s$ required to go from this point to $(a,y)$. Since $Y_{(a,y)}$ (green line) is known, the inversion is direct. 
	\begin{figure}[ht]
		\centering			
		\tikzset{every picture/.style={line width=0.75pt}} 
		
		\begin{tikzpicture}[x=0.7pt,y=0.7pt,yscale=-1,xscale=1]
		
		\draw  [color={rgb, 255:red, 128; green, 128; blue, 128 }  ,draw opacity=1 ][fill=yellow!10 ][line width=0.75]  (176.67,205.83) -- (242.33,205.83) -- (242.33,253.5) -- (176.67,253.5) -- cycle ;
		\draw  (149,260.83) -- (549,260.83)(167,19.5) -- (167,266.5) (542,255.83) -- (549,260.83) -- (542,265.83) (162,26.5) -- (167,19.5) -- (172,26.5)  ;
		\draw [color={rgb, 255:red, 155; green, 155; blue, 155 }  ,draw opacity=1 ]   (204.5,129.75) .. controls (187.83,135.08) and (178.17,137.25) .. (171.5,139.25) ;
		\draw [shift={(171.5,139.25)}, rotate = 343.3] [color={rgb, 255:red, 155; green, 155; blue, 155 }  ,draw opacity=1 ][line width=0.75]      (5.59,-5.59) .. controls (2.5,-5.59) and (0,-3.09) .. (0,0) .. controls (0,3.09) and (2.5,5.59) .. (5.59,5.59) ;
		\draw [color={rgb, 255:red, 128; green, 128; blue, 128 }  ,draw opacity=1 ]   (356.5,81.75) .. controls (372,74.75) and (378,70.25) .. (394,58.39) ;
		\draw [shift={(394,58.39)}, rotate = 143.45] [color={rgb, 255:red, 128; green, 128; blue, 128 }  ,draw opacity=1 ][line width=0.75]      (5.59,-5.59) .. controls (2.5,-5.59) and (0,-3.09) .. (0,0) .. controls (0,3.09) and (2.5,5.59) .. (5.59,5.59) ;
		\draw    (149.41,152.29) -- (163.74,152.63) ;
		\draw  [fill={rgb, 255:red, 0; green, 0; blue, 0 }  ,fill opacity=1 ] (398.83,61.83) .. controls (398.83,59.44) and (400.77,57.5) .. (403.17,57.5) .. controls (405.56,57.5) and (407.5,59.44) .. (407.5,61.83) .. controls (407.5,64.23) and (405.56,66.17) .. (403.17,66.17) .. controls (400.77,66.17) and (398.83,64.23) .. (398.83,61.83) -- cycle ;
		\draw  [fill={rgb, 255:red, 74; green, 144; blue, 226 }  ,fill opacity=1 ] (162.83,258.17) .. controls (162.83,255.77) and (164.77,253.83) .. (167.17,253.83) .. controls (169.56,253.83) and (171.5,255.77) .. (171.5,258.17) .. controls (171.5,260.56) and (169.56,262.5) .. (167.17,262.5) .. controls (164.77,262.5) and (162.83,260.56) .. (162.83,258.17) -- cycle ;
		\draw [color={rgb, 255:red, 74; green, 144; blue, 226 }  ,draw opacity=1 ]   (338,201.83) .. controls (299.2,231.35) and (287.29,291.89) .. (172.73,263.27) ;
		\draw [shift={(171,262.83)}, rotate = 14.33] [color={rgb, 255:red, 74; green, 144; blue, 226 }  ,draw opacity=1 ][line width=0.75]    (10.93,-3.29) .. controls (6.95,-1.4) and (3.31,-0.3) .. (0,0) .. controls (3.31,0.3) and (6.95,1.4) .. (10.93,3.29)   ;
		\draw [line width=1.5]  [dash pattern={on 1.69pt off 2.76pt}]  (338,196.44) -- (338.33,268.33) ;
		\draw [color={rgb, 255:red, 65; green, 117; blue, 5 }  ,draw opacity=1 ]   (75,163.75) .. controls (280.2,137.4) and (346.2,117.4) .. (403.17,61.83) ;
		\draw [color={rgb, 255:red, 208; green, 2; blue, 27 }  ,draw opacity=1 ]   (215.67,238.33) .. controls (209.4,237.8) and (349,239) .. (425.4,101.4) ;
		\draw [shift={(425.4,101.4)}, rotate = 119.04] [color={rgb, 255:red, 208; green, 2; blue, 27 }  ,draw opacity=1 ][line width=0.75]    (17.64,-3.29) .. controls (13.66,-1.4) and (10.02,-0.3) .. (6.71,0) .. controls (10.02,0.3) and (13.66,1.4) .. (17.64,3.29)(10.93,-3.29) .. controls (6.95,-1.4) and (3.31,-0.3) .. (0,0) .. controls (3.31,0.3) and (6.95,1.4) .. (10.93,3.29)   ;
		\draw [color={rgb, 255:red, 128; green, 128; blue, 128 }  ,draw opacity=1 ]   (283.5,212.75) .. controls (299,205.75) and (315.5,197.75) .. (331.5,185.89) ;
		\draw [shift={(331.5,185.89)}, rotate = 143.45] [color={rgb, 255:red, 128; green, 128; blue, 128 }  ,draw opacity=1 ][line width=0.75]      (5.59,-5.59) .. controls (2.5,-5.59) and (0,-3.09) .. (0,0) .. controls (0,3.09) and (2.5,5.59) .. (5.59,5.59) ;
		\draw [color={rgb, 255:red, 155; green, 155; blue, 155 }  ,draw opacity=1 ]   (265.5,220.25) .. controls (248.83,225.58) and (230.5,228.75) .. (221.5,229.25) ;
		\draw [shift={(221.5,229.25)}, rotate = 356.82] [color={rgb, 255:red, 155; green, 155; blue, 155 }  ,draw opacity=1 ][line width=0.75]      (5.59,-5.59) .. controls (2.5,-5.59) and (0,-3.09) .. (0,0) .. controls (0,3.09) and (2.5,5.59) .. (5.59,5.59) ;
		\draw  [fill={rgb, 255:red, 248; green, 231; blue, 28 }  ,fill opacity=1 ] (166,146.71) -- (167.4,149.76) -- (170.52,150.25) -- (168.26,152.63) -- (168.79,155.98) -- (166,154.4) -- (163.21,155.98) -- (163.74,152.63) -- (161.48,150.25) -- (164.6,149.76) -- cycle ;
		\draw  [fill={rgb, 255:red, 248; green, 231; blue, 28 }  ,fill opacity=1 ] (338,191.31) -- (339.4,194.36) -- (342.52,194.85) -- (340.26,197.23) -- (340.79,200.58) -- (338,199) -- (335.21,200.58) -- (335.74,197.23) -- (333.48,194.85) -- (336.6,194.36) -- cycle ;
		\draw  [fill={rgb, 255:red, 0; green, 0; blue, 0 }  ,fill opacity=1 ] (215.67,238.33) .. controls (215.67,235.94) and (217.61,234) .. (220,234) .. controls (222.39,234) and (224.33,235.94) .. (224.33,238.33) .. controls (224.33,240.73) and (222.39,242.67) .. (220,242.67) .. controls (217.61,242.67) and (215.67,240.73) .. (215.67,238.33) -- cycle ;
		\draw [color={rgb, 255:red, 208; green, 2; blue, 27 }  ,draw opacity=1 ]   (383.55,158.95) -- (367.45,174.25) ;
		\draw [shift={(367.45,174.25)}, rotate = 136.46] [color={rgb, 255:red, 208; green, 2; blue, 27 }  ,draw opacity=1 ][line width=0.75]    (10.93,-3.29) .. controls (6.95,-1.4) and (3.31,-0.3) .. (0,0) .. controls (3.31,0.3) and (6.95,1.4) .. (10.93,3.29)   ;
		\draw [shift={(385,157.57)}, rotate = 136.46] [color={rgb, 255:red, 208; green, 2; blue, 27 }  ,draw opacity=1 ][line width=0.75]    (10.93,-3.29) .. controls (6.95,-1.4) and (3.31,-0.3) .. (0,0) .. controls (3.31,0.3) and (6.95,1.4) .. (10.93,3.29)   ;
		\draw [color={rgb, 255:red, 208; green, 2; blue, 27 }  ,draw opacity=1 ]   (278.41,225.37) -- (265.76,228.93) ;
		\draw [shift={(265.76,228.93)}, rotate = 164.29] [color={rgb, 255:red, 208; green, 2; blue, 27 }  ,draw opacity=1 ][line width=0.75]    (10.93,-3.29) .. controls (6.95,-1.4) and (3.31,-0.3) .. (0,0) .. controls (3.31,0.3) and (6.95,1.4) .. (10.93,3.29)   ;
		\draw [shift={(280.33,224.83)}, rotate = 164.29] [color={rgb, 255:red, 208; green, 2; blue, 27 }  ,draw opacity=1 ][line width=0.75]    (10.93,-3.29) .. controls (6.95,-1.4) and (3.31,-0.3) .. (0,0) .. controls (3.31,0.3) and (6.95,1.4) .. (10.93,3.29)   ;
		\draw [color={rgb, 255:red, 65; green, 117; blue, 5 }  ,draw opacity=1 ]   (221.36,141.48) -- (205.31,143.99) ;
		\draw [shift={(205.31,143.99)}, rotate = 171.11] [color={rgb, 255:red, 65; green, 117; blue, 5 }  ,draw opacity=1 ][line width=0.75]    (10.93,-3.29) .. controls (6.95,-1.4) and (3.31,-0.3) .. (0,0) .. controls (3.31,0.3) and (6.95,1.4) .. (10.93,3.29)   ;
		\draw [shift={(223.33,141.17)}, rotate = 171.11] [color={rgb, 255:red, 65; green, 117; blue, 5 }  ,draw opacity=1 ][line width=0.75]    (10.93,-3.29) .. controls (6.95,-1.4) and (3.31,-0.3) .. (0,0) .. controls (3.31,0.3) and (6.95,1.4) .. (10.93,3.29)   ;
		\draw [color={rgb, 255:red, 65; green, 117; blue, 5 }  ,draw opacity=1 ]   (298.74,122.37) -- (284.26,126.42) ;
		\draw [shift={(284.26,126.42)}, rotate = 164.37] [color={rgb, 255:red, 65; green, 117; blue, 5 }  ,draw opacity=1 ][line width=0.75]    (10.93,-3.29) .. controls (6.95,-1.4) and (3.31,-0.3) .. (0,0) .. controls (3.31,0.3) and (6.95,1.4) .. (10.93,3.29)   ;
		\draw [shift={(300.67,121.83)}, rotate = 164.37] [color={rgb, 255:red, 65; green, 117; blue, 5 }  ,draw opacity=1 ][line width=0.75]    (10.93,-3.29) .. controls (6.95,-1.4) and (3.31,-0.3) .. (0,0) .. controls (3.31,0.3) and (6.95,1.4) .. (10.93,3.29)   ;
		\draw [color={rgb, 255:red, 65; green, 117; blue, 5 }  ,draw opacity=1 ]   (381.04,81.33) -- (372.63,87.34) ;
		\draw [shift={(372.63,87.34)}, rotate = 144.46] [color={rgb, 255:red, 65; green, 117; blue, 5 }  ,draw opacity=1 ][line width=0.75]    (10.93,-3.29) .. controls (6.95,-1.4) and (3.31,-0.3) .. (0,0) .. controls (3.31,0.3) and (6.95,1.4) .. (10.93,3.29)   ;
		\draw [shift={(382.67,80.17)}, rotate = 144.46] [color={rgb, 255:red, 65; green, 117; blue, 5 }  ,draw opacity=1 ][line width=0.75]    (10.93,-3.29) .. controls (6.95,-1.4) and (3.31,-0.3) .. (0,0) .. controls (3.31,0.3) and (6.95,1.4) .. (10.93,3.29)   ;
		
		\draw (223.67,238.73) node [anchor=north west][inner sep=0.75pt]  [font=\normalsize]  {$\mathbf{x}$};
		\draw (425.67,113.4) node [anchor=north west][inner sep=0.75pt]  [color={rgb, 255:red, 208; green, 2; blue, 27 }  ,opacity=1 ]  {$\Gamma _{\mathbf{x}}^{+}$};
		\draw (412,52.4) node [anchor=north west][inner sep=0.75pt]    {$( a,y)$};
		\draw (162,271.4) node [anchor=north west][inner sep=0.75pt]    {$0$};
		\draw (209.68,120.51) node [anchor=north west][inner sep=0.75pt]  [color={rgb, 255:red, 128; green, 128; blue, 128 }  ,opacity=1 ,rotate=-342.69,xslant=-0.05]  {$t-s=\phi _{( 0,Y_{( a,y)}( 0))}^{-1}( a,y)$};
		\draw (72.67,138.23) node [anchor=north west][inner sep=0.75pt]    {$z = Y_{( a,y)}( 0)$};
		\draw (556.67,262.07) node [anchor=north west][inner sep=0.75pt]    {$a \textrm{ (\textit{age}) }$};
		\draw (139.33,0.4) node [anchor=north west][inner sep=0.75pt]    {$y \textrm{ (\textit{size}) }$};
		\draw (268.07,214.19) node [anchor=north west][inner sep=0.75pt]  [color={rgb, 255:red, 128; green, 128; blue, 128 }  ,opacity=1 ,rotate=-344.11,xslant=-0.05]  {$s$};
		\draw (341.4,197.76) node [anchor=north west][inner sep=0.75pt]    {$\varphi ^{s}(\mathbf{x})$};
		\draw (325.33,115) node [anchor=north west][inner sep=0.75pt]  [color={rgb, 255:red, 65; green, 117; blue, 5 }  ,opacity=1 ]  {$\Gamma _{( a,y)}^{-}$};
		\draw (178.67,209.23) node [anchor=north west][inner sep=0.75pt]    {$\mathscr{K} $};

		\end{tikzpicture}
		
		\caption{Graphical description of the change of variables defined by $\gamma_t$}
		\label{fig:diffeo_lconstant}
	\end{figure}
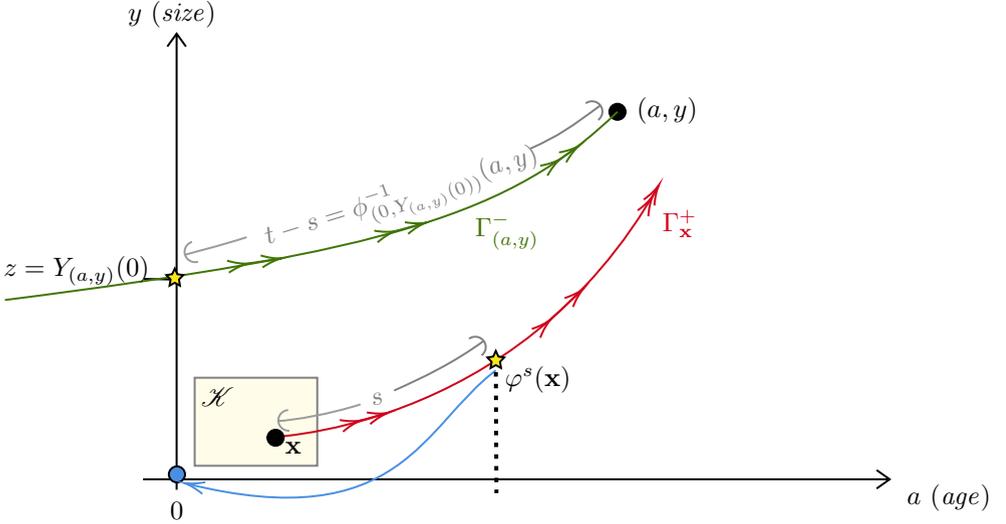
	We conclude that $\gamma_{t}$ is a $C^1$-diffeomorphism and then performing the change of variables $(a,y) = \gamma_{t}(s,z)$ in the RHS of Eq. \eqref{eq:lastDuhamelIter} gives
	\begin{align}
	P_t f(\bx) \geq  \int_{\rr_+^2}
	f(a,y) \Bigg\{  & 
	\ \psi\pars{t - \phi^{-1}_{0, Y_{(a,y)}(0)}(a,y) \vert  \bx }  \nonumber \\ 
	& \Psitx{\phi^{-1}_{0, Y_{(a,y)}(0)}(a,y)}{ (0, Y_{(a,y)}(0)) }   \nonumber \\
	& \frac{h(0, Y_{(a,y)}(0)) k \pars{\varphi^{t - \phi^{-1}_{0, Y_{(a,y)}(0)}(a,y)}(\bx), Y_{(a,y)}(0) } }{\displaystyle \int_0^\infty h(0, z) k \pars{\varphi^{t - \phi^{-1}_{0, Y_{(a,y)}(0)}(a,y)}(\bx), z } dz } \nonumber \\
	& \frac{1}{\abs{\det \mathcal{D} \gamma_t \pars{\gamma_t^{-1}(a,y)} }}  \mathds{1}_{\phi^{-1}_{0, Y_{(a,y)}(0)}(a,y) \leq t} \Bigg\} \ da \ dy . 
	\label{eq:Doeblin_justbeforebounds}
	\end{align}
	Now, using Assumptions~\ref{ass:all}, we can bound the functions and the Jacobian found in the obtained integral. First, since $g \geq 0$, note that $\norm{\varphi^t(\mathbf{x})} \geq \norm{\varphi^s(\mathbf{x})} $ for all $t > s$. Second, $\beta_- g_1(\mathbf{x}) \leq \beta(\mathbf{x}) \leq \beta_+ g_1(\mathbf{x}) $. And third, by the definition of the flow, $\int_0^t g_1 \pars{\varphi^s \pars{\mathbf{x}}} = \varphi^t_1(\mathbf{x})$ which equals the age at time $t$ of an individual with trait $\mathbf{x}$ at time 0.  Then, recalling that $\mathscr{K} \subset [\underline a,\bar a] \times [\underline y, \bar y]$, for all $t > 0$ we obtain the following bounds:
	\begin{enumerate}[i.]
		\item For all $(a_0,y_0) \in \mathscr{K}$, using the superior bounds on $g_1$ from Assumptions~\ref{ass:all}-i. we have
		\[
		\varphi^t_1(a_0, y_0) = a_0 + \int_0^t g_1(\varphi^s (a_0,y_0)) ds \leq a_0 +  \int_0^t c_1(1 + \varphi^s_1(a_0,y_0)) ds 
		\]
		Hence, by Gronwall inequality
		\[
		\varphi^t_1(a_0, y_0)  \leq \pars{a_0 + c_1 t}e^{c_1 t} \leq \pars{\bar a + c_1 t}e^{c_1 t}
		\]
		Analogously, using the lower bounds on $g_1$ from Assumptions~\ref{ass:all}-i., we obtain
		\[
		\varphi^t_1(a_0, y_0)  \geq a_0 e^{c_0 t} \geq \underline{a} e^{c_0 t}
		\]
		\item From the previous result, for all $\bx \in \mathscr{K}$
		\begin{align*}\Psitx{t}{\bx} \geq \exp \pars {- \beta_+ \int_0^t g_1(\varphi^{s} (\mathbf{x})) ds } &= e^{ - \beta_+ \varphi^t_1(\mathbf{x})}  \\ &\geq  e^{ - \beta_+ \pars{\bar a + c_1 t}e^{c_1 t}} , \end{align*}
		\item Analogously
		\[\beta(\varphi^t (\bx) )  \geq \beta_- g_1(\varphi^{t} (\mathbf{x}))   \geq \beta_- c_0 \varphi^t_1(\mathbf{x}) \geq \beta_- c_0 \underline{a} e^{c_0 t}. \]
		Therefore there are some constants $A_0, B_0 > 0$ such that
		\begin{equation}
		\psi\pars{t - s \vert  \bx } \geq  A_0 \exp \pars{- B_0 (1 + t - s) e^{c_1 (t-s)} }
		\label{eq:boundA0B0}
		\end{equation}
		\item Moreover, recall that the eigenfunction $h$ is solution to Eq. \eqref{eq:renewal}. Then, by Fubini's Theorem, for all $\bx \in \mathscr{K}$,
		\[
		h(\bx) = \int_0^\infty \pars{\int_0^\infty h(0,z) k \pars{\varphi^t(\bx),z } dz} \psi(t\vert \bx) e^{-\lambda t} dt .
		\]
		Thus, in particular, the integrability gives us that 
		\[
		\pars{\int_0^\infty h(0,z) k \pars{\varphi^t(\bx),z } dz} \frac{\psi(t\vert \bx) e^{-\lambda t} }{h(\bx)}  \to 0 \quad \textrm{ as } t \to +\infty.
		\]
		Therefore, there exist some constants $C_1, C_2 > 0$ such that for all $\bx \in \mathscr{K}$ there is some time $T(\bx) > 0$ such that for all $t > 0$ we have
		\[
		\pars{\int_0^\infty h(0,z) k \pars{\varphi^t(\bx),z } dz} \frac{\psi(t\vert \bx) e^{-\lambda t} }{h(\bx)} \leq  C_1 \mathds{1}_{t \geq T(\bx)} + C_2 \mathds{1}_{t < T(\bx)}
		\]
		where
		\[
		\sup_{\bx \in \mathscr{K}} \sup_{t < T(\bx)}  \pars{\int_0^\infty h(0,z) k \pars{\varphi^t(\bx),z } dz} \frac{\psi(t\vert \bx) e^{-\lambda t} }{h(\bx)} \leq C_2 ,
		\]
		since the suprema are taken in a compact set and for a continuous locally bounded function.
		Then, taking $C_0 = \max \setof{C_1,C_2}$ we have
		\[
		\frac{1}{\int_0^\infty h(0, z) k \pars{\varphi^{t} (\bx), z } dz} \geq \frac{\psi(t\vert \bx) e^{-\lambda t} }{C_0 h(\bx)} 
		\]
		where $\psi(t\vert \bx)$ can again be bounded by below using Eq. \eqref{eq:boundA0B0}. Moreover, the continuity of $h$ implies that $h$ is locally bounded and hence, for all $\bx \in \mathscr{K}$, $h(\bx) \leq H_0 < \infty$. Hence we obtain finally
		\begin{equation}
		\frac{1}{\int_0^\infty h(0, z) k \pars{\varphi^{t} (\bx), z } dz} \geq \frac{A_0}{C_0 H_0}  \exp \pars{-B_0 (1+t)e^{c_1 t}- \lambda t} 
		\label{eq:boundC0}
		\end{equation}			
		Note that these three estimates give bounds which are dependent only on $t$. 
		\item From \eqref{eq:JacobianMatrix}, for all $s \leq t$ and $z > 0$, the Jacobian determinant equals
		\begin{equation}
		\det \mathcal{D} \gamma_{t} (s,z) = \norm{g\pars {\varphi^{t-s}(0,z)}} \norm{\nabla_{z} \varphi^{t-s}(0,z) } \sin \theta(s,t,z)
		\label{eq:detD_casConstant}
		\end{equation}
		where $\theta(s,t,z)$ is the angle between $g\pars {\varphi^{t-s}(0,z)}$ and $\nabla_{z} \varphi^{t-s}(0,z) $. Hence, from Lemma~\ref{lemma:flowprop}.2, we get
		\begin{align*}
		\abs{\det \mathcal{D} \gamma_t \pars{s,z} } &\leq \norm{g\pars {\varphi^{t-s}(0,z)}} \ \mnorm{\mathcal{D} \varphi^{t-s} \pars {0,z}}, 
		\end{align*}
		where $\mnorm{\cdot}$ is the matrix norm induced by $\norm{\cdot}$, and therefore
		\begin{align}
		\abs{\det \mathcal{D} \gamma_t \pars{\gamma_t^{-1}(a,y)} } &\leq  \norm{g(a,y)} 
		\mnorm{\mathcal{D} \varphi^{\phi^{-1}_{0, Y_{(a,y)}(0)}(a,y)} \pars{0,Y_{(a,y)}(0)}}  \nonumber \\
		&=: \norm{g\pars {a,y}}  E_0(a,y)  \label{eq:boundDet}
		\end{align}
		Note that this bound depends only on $(a,y)$ and neither on $\bx$ or $t$.
	\end{enumerate}
	Hence, coming back to Eq. \eqref{eq:Doeblin_justbeforebounds} and applying the bounds \eqref{eq:boundA0B0}, \eqref{eq:boundC0} and \eqref{eq:boundC0} to the integrands, we obtain
	\begin{align}
	P_t f(\bx) \geq  \int_{\rr_+^2}
	f(a,y) \Bigg\{  & 
	\ \frac{A_0^2}{C_0 H_0} \exp \pars{- 2 B_0 (1 + t) e^{c_1 t} - \lambda t } \nonumber \\
	& \Psitx{\phi^{-1}_{0, Y_{(a,y)}(0)}(a,y)}{0, Y_{(a,y)}(0) }   \nonumber \\ 
	& h(0, Y_{(a,y)}(0)) k \pars{\varphi^{t - \phi^{-1}_{0, Y_{(a,y)}(0)}(a,y)}(\bx), Y_{(a,y)}(0) }  \nonumber \\
	& \frac{1}{\norm{g\pars {a,y}}  E_0(a,y)} \mathds{1}_{\phi^{-1}_{0, Y_{(a,y)}(0)}(a,y) \leq t} \Bigg\} \ da \ dy . 
	\label{eq:Doeblin_justafterbounds}
	\end{align}
	Now, we make use of the petite-set condition which allows us to average the value of $P_t f(\bx)$ against a discrete sampling measure $\mu(dt)$ over a $\Delta$-skeleton. This is, consider some $\Delta > 0$, which will be fixed later on, and a measure $\mu$ over $\setof{j \Delta : j \in \mathbb{N}}$, characterised by a sequence $(\mu_j)_{j \in \mathbb{N}}$ with $\sum \mu_j = 1$ and $\mu_j > 0$ for all $j \in \mathbb{N}$. We have
	\begin{align*}
	\ap{\mu}{\delta_{\bx} P_{\cdot}f}  \geq \sum_{j=0}^{\infty} \mu_j \int_{\cx}
	f(a,y)   & \ k \pars{\varphi^{j \Delta - \phi^{-1}_{0, Y_{(a,y)}(0)}(a,y)}(\mathbf{x}), Y_{(a,y)}(0) }  \\
	& \ \zeta(a,y) e^{- \tilde \beta j \Delta e^{j \Delta} }   \mathds{1}_{\phi^{-1}_{0, Y_{(a,y)}(0)}(a,y) \leq j \Delta } \ da \ dy,
	\end{align*}
	where the the function $\zeta(a,y)$ is constructed by regrouping all the terms which depend only on $(a,y)$ (and neither on $\bx$ or $t$), and the constant $\tilde \beta > 0$ is obtained after selecting only the dominant term inside the exponential.
	Now, it remains to loose the dependency on $\mathbf{x}$ using that $\mathbf{x} \in \mathscr{K}$ to find a uniform lower bound for the whole compact. By Assumption~\ref{ass:all}-(iv), we have that for all $z$, exists $D(z) \subset \rr_+$ such that $k(\varphi^s(\mathbf{x}), z) > \varepsilon(z) \mathds{1}_{D(z)}(\varphi^s(\mathbf{x}))$. Then, let
	\[
	\mathcal{T}(\mathbf{x},z) := \setof{s>0 : \varphi^s(\mathbf{x}) \in D(z) } ,
	\]
	then
	\[
	k(\varphi^s(\mathbf{x}), z) > \varepsilon(z) \mathds{1}_{\mathcal{T}(\mathbf{x},z) }(s).
	\]
	Now, let $\Delta = \inf_{z>0} \textrm{diam} (D(z)) > \delta_- > 0$. Then, for all $\mathbf{x} \in \mathscr{K}$ and $z > 0$ there exists $n = n(\bx,z) \in \mathbb{N}$ such that $n \Delta \in \mathcal{T}(\mathbf{x},z) $. Then for all $\bx \in \mathscr{K}$ and $z > 0$, 
	\[
	\sum_{j=0}^\infty \mathds{1}_{j \Delta \in \mathcal{T}(\mathbf{x},z) } \geq 1.
	\]
	Moreover, since for all $z>0$, $\textrm{diam} (D(z)) < \delta_+$, there exists some $j$ big enough such that the trajectory leaves $D(z)$. In particular, the compactness of $\mathscr{K}$ implies that it exists $j^*$ such that for all $\mathbf{x} \in \mathscr{K}$
	\[
	\mathds{1}_{j \Delta \in \mathcal{T}(\mathbf{x},z) } = 0 \ \  \forall j \geq j^*.
	\]
	Therefore for any sampling measure $(\mu_j)_j$ we have 
	\[
	\sum_{j=0}^\infty \mu_j \mathds{1}_{j \Delta \in \mathcal{T}(\mathbf{x},z) } \geq \min_{j \leq j^*} \mu_j  ,
	\]
	and finally for all fixed $\tau > 0$,
	\begin{align*}
	\sum_{j=0}^{\infty} \mu_j  e^{- \tilde \beta j \Delta e^{j \Delta} } k \pars{\varphi^{j \Delta - \tau }(\mathbf{x}), z } \mathds{1}_{\tau \leq j \Delta} &\geq \sum_{j=0}^{\infty} \mu_j  e^{- \tilde \beta j \Delta e^{j \Delta} } \epsilon(z) \mathds{1}_{j \Delta - \tau \in \mathcal{T}(\mathbf{x},z) }  \\
	&\geq \epsilon(z) \min_{j \leq j^*} \mu_j \ \min_{j \leq j^*} e^{- \tilde \beta j \Delta e^{j \Delta}} 
	\end{align*}
	from what we can conclude that
	\[
	\ap{\mu}{\delta_{\bx} P_{\cdot}f}  \geq \int_\cx f(a,y) \nu(a,y) da dy
	\]
	with
	\[
	\nu(a,y) = \zeta(a,y) \epsilon \pars{ Y_{(a,y)}(0)}  e^{- \tilde \beta j^* e^{\Delta j^* \Delta}} \min_{j \leq j^*} \mu_j
	\]
\end{proof}	
Finally, the proof of the main theorem~\ref{thm:main} is a direct application of Harris Theorem~\ref{thm:harris}. 		

\section{Application: Steady-state size distribution of the adder model of bacterial proliferation}	
\label{sec:adder}
We recall the generator of the adder model of \textit{E. coli} growth introduced in Example~\ref{ex:adder}:
\begin{align*}
\mathcal{Q}f (a,y)
=& \lambda y \pars{ \partial_a + \partial_y }f(a , y) \nonumber \\
&+ \lambda y B(a)  \pars{ 2 \int_0^1 f(0,\rho y)  F(\rho) d\rho  -f(a,y)}  - d_0 f(a,y). 
\end{align*}
We assume that:
\begin{assumptions}
	\label{ass:1}
	Suppose
	\begin{enumerate}[label=(A\arabic*)]
		\item There exist $0 < \underline{b} \leq \bar{b} < \infty$ such that for all $a \geq 0$, $\underline{b} < B(a) < \bar{b}$.
		\item $F$ is a continuous positive function in $[0,1]$, with connected support. We call for all $k \geq 0$,
		\[
		m_k = \int_0^1 \rho^k F(\rho) d \rho
		\]
		and suppose that $m_0 = 1$, $m_1= 1/2$ and $m_2 < +\infty$. Note that, since $\rho \in (0,1)$ almost surely, then for all $k>0$ we have $m_k \leq m_1 = 1/2$. 
		\item $\lambda > d_0$.
	\end{enumerate}
\end{assumptions}

\begin{remark}[Doob $h$-transform.] { 
In this case, it is straightforward to verify that $h(a,y) = y$ is an eigenfunction of $\mathcal{Q}$ associated to eigenvalue $\Lambda = \lambda - d_0$. In particular, from Eq. \eqref{eq:A} the Doob $h$-transformed semigroup $P_t$ is generated by the conservative infinitesimal generator
\begin{align}
\mathcal{A}f (a,y) = \lambda y \pars{ \partial_a + \partial_y }f(a , y)
&+ 2 \lambda y B(a)   \int_0^1 \pars {f(0,\rho y)   - f(a,y)} \rho F(\rho) d\rho  .
\label{eq:A_adder}
\end{align}
Indeed, the rescaled kernel gives
\[
\frac{h(0,z)}{h(a,y)} k(a,y,z) = 2 \frac{z}{y} \cdot \frac{1}{y} F \pars{\frac{z}{y}} \mathds{1}_{z \leq y} ,
\]
which under the change of variables $z \mapsto \rho = z/y$ gives the probability density $\rho \mapsto 2 \rho F(\rho)$ supported in $[0,1]$, which is indeed a probability by Assumption (A2). 
}
\end{remark}

Then, we have the following result of exponential convergence, that completes the analysis started by~\cite{gabriel:hal-01742140} from an operator theory approach, where exponential convergence could not be obtained from the estimates of relative entropy.

\begin{theorem}
	Under Assumptions (A1)-(A3), there is a unique probability measure $\pi^*$ such that there exist constants $C, \omega > 0$ which verify Eq. \eqref{eq:exponCV}  with $\Lambda = \lambda - d_0$, $h(a,y)=y$ and $V(a,y) = y^{-1} + y$. Moreover $\pi^*$ admits a density given explicitly by
	\[
	\pi^*(a,y) = \frac{\exp \pars{- \int_0^a B(\alpha) d\alpha } }{y^2} \eta^* (y-a) ,
	\]
	where $\eta^*$ is the unique solution to the fixed point problem
	\[
	\eta^*(y) =  2 \int_0^1 \setof{ \int_0^{\frac{x}{\rho}}  \psi \pars{\frac{y}{\rho} - z}    \eta^* \pars{z}   dz } F(\rho)d \rho,
	\]
	where
	$
	\psi(a) = B(a)  \exp \pars{- \int_0^a B(\alpha) d\alpha } .
	$
	\label{thm:main_adder}
\end{theorem}
\begin{proof}
	\begin{enumerate}
		\item \textbf{minorisation condition.} It is a direct application of Proposition~\ref{lemma:doeblinmin}, since the same hypothesis in Assumptions~\ref{ass:all} are verified by Assumptions~\ref{ass:1}. Assumption~\ref{ass:1}-(v) requires some attention. Indeed, since $F$ is bounded and with connected support, $k(a,y,z) = \frac{1}{y} F\pars{\frac{z}{y}} \mathds{1}_{z \leq y}$ can be lower bounded in the form $k(a,y,z) >\varepsilon(z) \mathds{1}_{y \in D(z)}$, with $\varepsilon(z)$ of order $1/z$, as represents the example of Fig.~\ref{fig:kAdderAssumps}.
		
		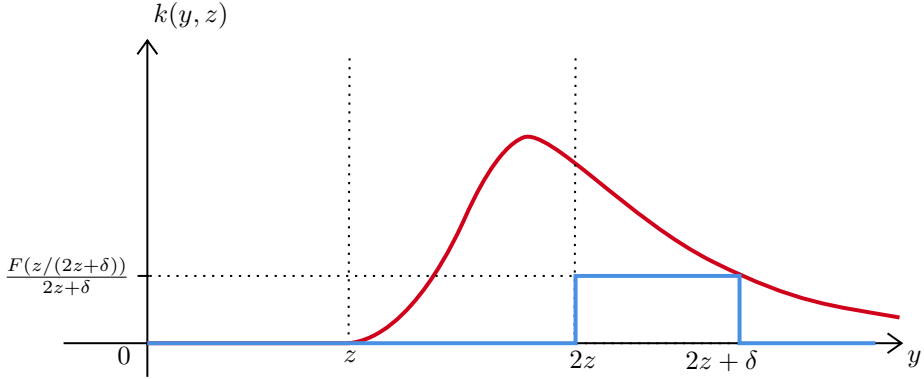
\begin{figure}[ht]
			\centering

			\tikzset{every picture/.style={line width=0.75pt}} 
			
			\begin{tikzpicture}[x=0.75pt,y=0.75pt,yscale=-1,xscale=1]
			
			\draw  (106.5,185) -- (525.5,185)(148.4,32) -- (148.4,202) (518.5,180) -- (525.5,185) -- (518.5,190) (143.4,39) -- (148.4,32) -- (153.4,39)  ;
			\draw [color={rgb, 255:red, 208; green, 2; blue, 27 }  ,draw opacity=1 ][line width=1.5]    (249,185) .. controls (280.5,182) and (303.5,130) .. (307.5,121) .. controls (311.5,112) and (323.5,86) .. (336.5,81) .. controls (349.5,76) and (389.5,123) .. (431.5,144) .. controls (473.5,165) and (490.5,166) .. (524.5,172) ;
			\draw  [dash pattern={on 0.84pt off 2.51pt}]  (249.5,41) -- (249,185) ;
			\draw [color={rgb, 255:red, 208; green, 2; blue, 27 }  ,draw opacity=1 ][line width=1.5]    (148.4,185) -- (249,185) ;
			\draw  [dash pattern={on 0.84pt off 2.51pt}]  (149.5,151) -- (362.5,151) ;
			\draw    (143.5,151) -- (149.5,151) ;
			\draw  [dash pattern={on 0.84pt off 2.51pt}]  (362.5,41) -- (362,185) ;
			\draw  [dash pattern={on 0.84pt off 2.51pt}] (444.5,151) -- (362.5,151) -- (362.5,185) -- (444.5,185) -- cycle ;
			\draw [color={rgb, 255:red, 74; green, 144; blue, 226 }  ,draw opacity=1 ][line width=1.5]    (148.4,185) -- (362.5,185) -- (362.5,151) -- (444.5,151) -- (444.5,185) -- (512.5,185) ;
			
			\draw (245,187.4) node [anchor=north west][inner sep=0.75pt]    {$z$};
			\draw (150,11.4) node [anchor=north west][inner sep=0.75pt]    {$k( y,z)$};
			\draw (527,187.4) node [anchor=north west][inner sep=0.75pt]    {$y$};
			\draw (75,140) node [anchor=north west][inner sep=0.75pt]    {$\frac{F( z/( 2z+\delta ))}{2z+\delta }$};
			\draw (358,187.4) node [anchor=north west][inner sep=0.75pt]    {$2z$};
			\draw (132,185.4) node [anchor=north west][inner sep=0.75pt]    {$0$};
			\draw (416,187.4) node [anchor=north west][inner sep=0.75pt]    {$2z+\delta $};

			\end{tikzpicture}
			
			\caption{Example of minorisation for $k(a,y,z) = \frac{1}{y} F\pars{\frac{z}{y}} \mathds{1}_{z \leq y}$ and $F$ given by the probability density function of a Beta distribution. Then we have $k(a,y,z) >\varepsilon(z) \mathds{1}_{y \in D(z)}$ as required by Assumption~\ref{ass:all}-(v), with $|D(z)| = \delta$ for all $z$.}
			\label{fig:kAdderAssumps}
		\end{figure}
		
		In general, we have for all $\delta > 0$:
		\[
		k(a,y,z) > \min_{z' \in [2z, 2z+\delta]} \frac{F(z/z')}{z'} \ \mathds{1}_{y \in [2z, 2z+\delta]}
		\]
		and we verify then Assumption~\ref{ass:all}-(v) with $\varepsilon(z) = \min_{z' \in [2z, 2z+\delta]} \frac{F(z/z')}{z'} $ and $D(z) = [2z, 2z+\delta]$ for a chosen $\delta > 0$.
		
		Fig.~\ref{fig:doeblinAdder} shows the characteristics curves $y-a = \textrm{constant}$, and the shadowed region corresponds the space that is a priori reachable from the initial point along trajectories with exactly one jump before time $t$. It is the version of Fig.~\ref{fig:diffeo_lconstant} in this specific case. Moreover, given an initial point (A in Fig.~\ref{fig:diffeo_lconstant}) and total trajectory time, this reachable region is compact, which also simplifies some minorisations. 
		
		Finally, depending on the choice of the compact set $\mathscr{K} \subset [\underline{a}, \bar a] \times [\underline{y}, \bar y]$ and of $\delta$ (which gives also the discretisation timeStep~of the $\delta$-skeleton), the value of the minorant measure $\nu$ can be computed explicitly by numerical approximations, as given in Fig.~\ref{fig:nus} for different forms of $F$.
		
		\begin{figure}[ht]
			\centering
			\begin{tikzpicture}
			\begin{axis}[
			clip=false,
			xmin=0,
			xmax=13,
			ymin=0,
			ymax=13,
			axis x line = bottom,
			axis y line = left,
			xtick={1,2,6},
			ytick={5,7,9,10},
			xticklabels={$a_0$,$a^*$,$a_0+y_0(e^{\lambda t}-1)$},
			yticklabels={$y_0$,,$y^*$,$y_0 e^{\lambda t}$},
			]
			\node[below right] at (13,0) {$a$};
			\node[above left] at (0,13) {$y$};
			\draw[thick] (0,0) -- (13,13);
			\node[below right] at (13,13) {$a=y$};
			\node at (2,12) {$\cx$};
			\draw[thick, dotted] (0,0) -- (6, 12);
			\node[above right] at (6,11) {$a=y(1-e^{-\lambda t})$};
			\draw[fill=gray!10, dashed]   (0,0) -- (5,10) -- (0,10) -- (0,0);%
			\draw[help lines, dotted] (1,0) -- (1,5) -- (0,5);%
			\draw[help lines, dotted] (6,0) -- (6,10) -- (0,10);%
			\draw[help lines, dotted] (2,0) -- (2,9) -- (0,9);%
			\draw[red, stealth reversed-{stealth}{stealth}] (1,5) -- (3,7) ;%
			\draw[red, -{stealth}{stealth}] (3,7) -- (5,9);%
			\draw[red, dashed] (5,9) -- (6,10);%
			\draw[fill] (1,5) circle [radius=0.1];
			\node[below right] at (1,5) {$A$};
			\draw[fill] (5,9) circle [radius=0.1];
			\node[below right] at (5,9) {$B$};
			\draw[fill] (0,7) circle [radius=0.1];
			\node[below right] at (0,7) {$C$};
			\draw[fill] (2,9) circle [radius=0.1];
			\node[below right] at (2,9) {$D$};
			\draw[red, -{stealth}{stealth}] (0,7) -- (2,9);%
			\draw[cyan, -stealth] (5,9) to [out=180,in=-15] (0,7);%
			\end{axis}
			\end{tikzpicture}
			
			\caption{Ideal trajectory from initial point $A=(a_0,y_0)$ to point $D=(a^*,y^*)$ in time $t$. The individual spends a time $t-s$ growing from $A$ to $B$. Then, it divides and renews at point $C$. Finally, it grows the remaining time $s$ until point $D$. }
			\label{fig:doeblinAdder}
		\end{figure}
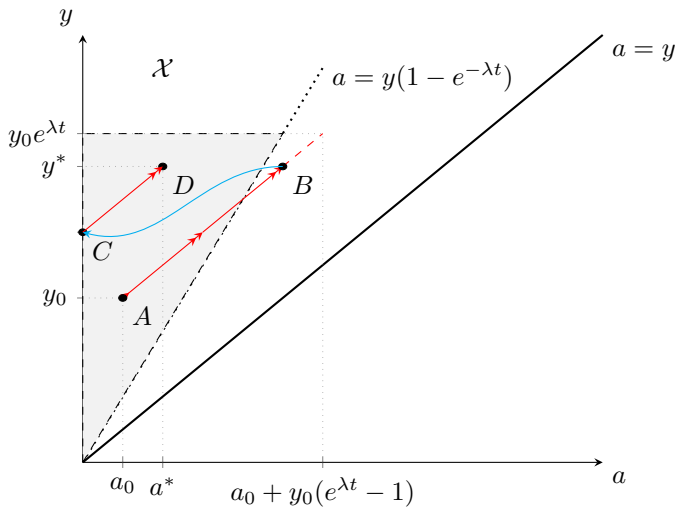			
		\item \textbf{Lyapunov-Foster condition.} { Consider the generator $\mathcal{A}$ defined by Eq. \eqref{eq:A_adder}}. Let $V(a,y) = y^{-1} + y$ with. It is clear that $V(a,y) \to \infty$ as $\vert (a,y)\vert  \to \infty$.
		Let $v(a,y) = y^k$, then
		\begin{align*}
		\mathcal{A} v(a,y) &= \lambda y^k + 2 \lambda y B(a) \int_0^1 \pars{\rho^k y^k - y^k} \rho F(\rho) d\rho \\
		&= \pars{k \lambda + 2 \pars{m_{k+1} - \frac{1}{2}} \lambda y B(a) } v(a,y)
		\end{align*}
		So, for $V(a,y) = y^{-1} + y$, we obtain
		\begin{align*}
		\mathcal{A} V(a,y) = &- \lambda  V(a,y) +  \varDelta(a,y),
		\end{align*}
		where
		\[
		\varDelta(a,y) := 2 \lambda y + 2 \lambda B(a) \pars{ \pars{\frac{1}{2} + \pars{m_{2} - \frac{1}{2}} y^2 }} .
		\]
		We already have $- \lambda < 0$ in the first term of the RHS. It remains to prove that $\varDelta(a,y)$ defined in the RHS above, is bounded.
		Indeed, notice that
		\[
		\varDelta(a,y) \leq 2 \lambda \pars{\bar b \pars{m_{2} - \frac{1}{2}} y^2 + y + \frac{\bar b}{2}}
		\]
		which is quadratic in $y$ with a negative quadratic coefficient since $m_{2} - 1/2 \leq 0$. Thus
		\begin{equation}
		\varDelta(a,y) \leq \lambda \pars{\bar b + \frac{1}{\bar b (1 - 2m_2)}} =:d \in \rr_+
		\end{equation}
		So finally we obtain that for every $(a,y) \in \rr_+^2$
		\[
		\mathcal{A} V(a,y) \leq - \lambda V(a,y) + d
		\]
		\item \textbf{Application of V-uniform Ergodic Theorem}
		Using Theorem~\ref{thm:harris} we conclude the existence of some $C, \omega > 0$ such that for every ${\bx} \in \mathcal{X}$ and $t \geq 0$
		\begin{equation}
		\norm{\delta_{\bx} P_t - \pi}_{V} \leq C  V(\bx) \exp(-\omega t) 
		\end{equation}
		Now, using that by construction, $M_t f = e^{\Lambda t} h P_t \pars{f/h}$, we obtain that for all $\bx \in \mathcal{X}$,
		\begin{equation}
		\norm{e^{-\Lambda t} \delta_{\bx} M_t - h(\bx) \pi^* }_{V} \leq C V(\bx) e^{-\omega t}
		\end{equation}
		where for every $A \in \mathcal{B}(\rr_+^2 \setminus \setof{0})$,
		\[
		\pi^*(A) = \int_A \frac{\pi(d\bx)}{h(\bx)}
		\]
		Moreover, we know that $h(a,y) = y$. On the other hand, $\pi$ is the unique solution to $\pi P_t = \pi$, or equivalently, to the dual eigenvalue problem associated to the conservative problem $\pi \mathcal{A} = 0$. From \eqref{eq:A} we obtain from the latter that $\pi$ is then the measure solution to the following PDE in the sense of distributions
		\begin{equation}
		\begin{cases}
		\begin{split}
		(\partial_a + \partial_y) (\lambda y \pi( a, y))  -  \lambda y B(a)  \pi (a,y)  &= 0  \\
		\pi(0,y) &=  2 \int_0^1 \int_0^\infty B(a) \frac{F(\rho)}{\rho} \pi \pars{a, \frac{y}{\rho}} da d \rho \\
		\int_0^\infty \int_0^y \pi(a,y) da dy &= 1
		\end{split}
		\end{cases}
		\label{eq:Eigenproblem:pi}
		\end{equation}
		We solve it by the method of characteristics. From the first equation of \eqref{eq:Eigenproblem:pi}, $\pi$ solves the ODE
		\[
		\begin{cases}
		\frac{d}{da} \pi(a,y(a)) = - \pars{B(a) + \frac{1}{y(a)}} \pi(a,y(a)) \\
		\pi(0,y(0)) = 2 \int_0^1 \int_0^\infty B(a) \frac{F(\rho)}{\rho} \pi \pars{a, \frac{y(0)-a}{\rho}} da d \rho
		\end{cases}
		\]
		where the associated characteristics are of the form $y(a) = a + (y(0)-a(0))$. Then, the solution $\pi$ of \eqref{eq:Eigenproblem:pi} is given by 
		\[
		\pi(a,y) = \frac{\exp \pars{- \int_0^a B(\alpha) d\alpha } }{y} \eta^* (y-a) ,
		\]
		where the definition of $\eta^*$ is inherited from the initial condition of the ODE:
		\[
		\pi(0,y(0)) = \frac{\eta^* (y-a) }{y-a} =2 \int_0^1 \int_0^\infty B(a) \frac{F(\rho)}{\rho} \pi \pars{a, \frac{y(0)-a}{\rho}} da d \rho.
		\]
		Note that the RHS still depends implicitly on $\pi$. Hence, $\eta^*$ is solution to the fixed point problem
		\begin{align*}
		\eta^*(x) &= 2 \int_0^1 \int_0^\infty B(a) F(\rho) \exp \pars{- \int_0^a B(\alpha) d\alpha }  \eta^* \pars{\frac{x}{\rho} - a}   da d \rho \\
		&=  2 \int_0^1 \int_0^{\frac{x}{\rho}} F(\rho) \psi \pars{\frac{x}{\rho} - a}    \eta^* \pars{a}   da d \rho,
		\end{align*}
		where
		\[
		\psi(a) = B(a)  \exp \pars{- \int_0^a B(\alpha) d\alpha } 
		\]
		is the probability density function of the added size at division.
		The existence of a formal solution to this problem is then a by-product of the existence of $\pi$, here provided by Harris' Theorem. 
		
		Thus finally, the stationary profile of $M_t$ is given by
		\[
		\pi^*(a,y) = \frac{\exp \pars{- \int_0^a B(\alpha) d\alpha } }{y^2} \eta^* (y-a) 
		\]
	\end{enumerate}
	\bigskip
\end{proof}

\begin{figure}[ht]
	\centering
	\includegraphics[width=0.99\textwidth]{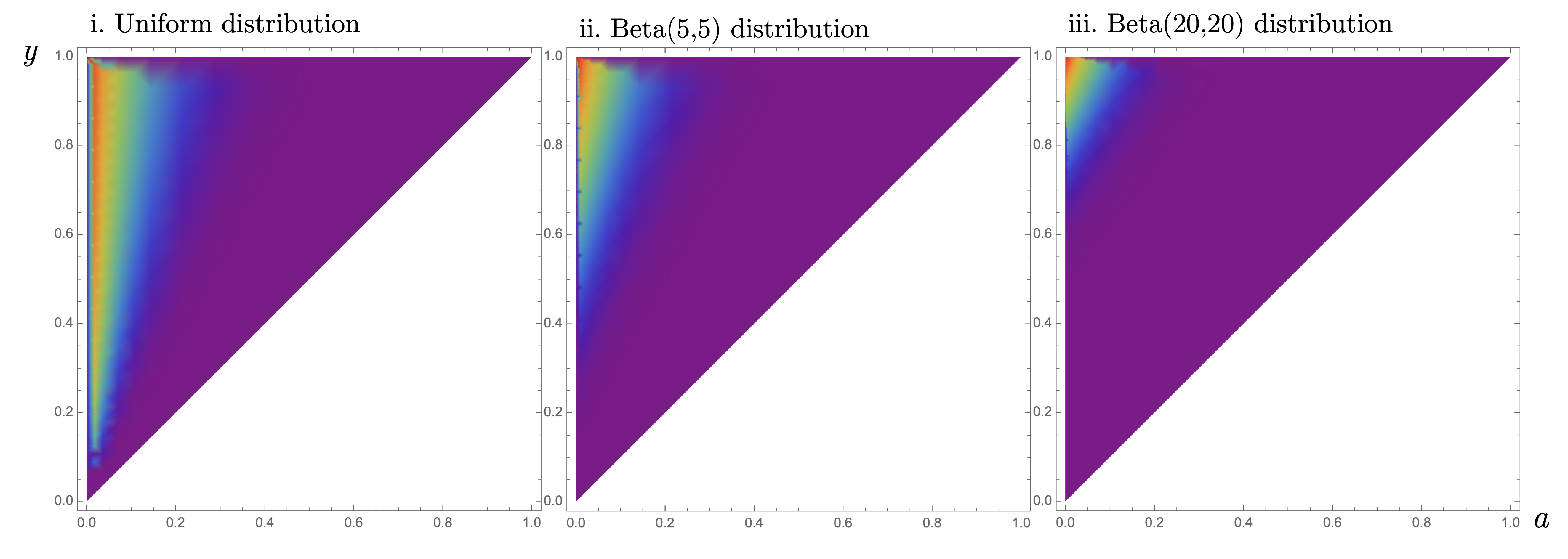}
	\caption{Minorant measure $\nu$ for $F$ given by \textit{i.} the uniform distribution, \textit{ii.} a Beta(5,5) distribution and \textit{iii.} a Beta(20,20) distribution. The values were obtained from numerical approximation. Only $F$ changes between the three cases.}
	\label{fig:nus}
\end{figure}

\begin{remark}
	The stability of this model has already been studied in the early works of~\cite{Hall1991SteadySD} for an application to plant physiology, and more recently by~\cite{gabriel:hal-01742140} where the exponential ergodicity could not be obtained using General Relative Entropy techniques. In our case however, the direct application of Theorem~\ref{thm:harris}, since the eigenelements of $\mathcal{Q}$ are known, allows to prove this result. More generally, when the drift term $g(a,y)$ is not necessarily given by the exponential elongation, the previous section allows to prove the existence of the suitable eigenelements. This was left as an open question by the works of~\cite{gabriel:hal-01742140}.	
\end{remark}

\begin{remark}
	The proof presented above does not work for singular divisions as given, in lieu of a density $F$, by $\rho$ distributed according to $\delta_{1/2}(d\rho)$ as in a perfectly symmetric mitosis. Indeed, the change of variables is no longer possible since $z$ would be constant. Moreover, if we try to pursue the method and average in time, one can check that the obtained $\nu$ would be the trivial measure for some large enough compacts. Such a limitation is not really surprising, since the authors of~\cite{Doumic2017} have already shown that if the elongation rate $\lambda$ is constant for the whole population (as in our case), and the divisions are perfectly symmetrical, then we do not have convergence, and a periodic behaviour is observed.
\end{remark}

\begin{remark} Figure~\ref{fig:nus} shows the shape of $\nu$ for different forms of the density $F$. As $F$ concentrates we can observe the increasing degeneracy of $\nu$. 
\end{remark}

\begin{remark}
	Here, the existence of $\eta^*$ is a by-product of the existence and uniqueness of $\pi$ provided by Harris' Theorem. In contrast, in the works of~\cite{gabriel:hal-01742140}, the existence of the stationary measure depended on the existence of a unique solution to the fixed point problem. Thus, the authors had to show compactness properties of the operator associated to the fixed point problem, which our approach evades. Moreover, our approach allows more general forms for $F$, which~\cite{gabriel:hal-01742140} requires to be of compact support strictly included in $]0,1[$.
\end{remark}

\section{Conclusions}

{

The present article studies spectral and ergodic properties of the first-moment semigroup of an age-size piecewise deterministic measure-valued process. The specificities discussed here are two-fold. 

First, we have that the process is non-conservative, and that the eigenelements are in general unkwnon. This is addressed by following a classical truncation scheme in order to apply Krein-Rutman's theorem. The key is to use the renewal property brought by the age structure that allows to reduce the dimension of the problem. Second, we have some sort of degeneracy arising from the age-coordinate jumps (reset at 0 at each reproduction event), along with a pure advection term between jumps, that makes it non-trivial to show mixing trajectories that explore the two-dimensional unbounded domain $\mathcal{X}$ independently with respect to the initial state. This is addressed by proving a minorisation condition for petite sets. This condition is seemingly weaker than more usual small set conditions, since it allows to average the action of the semigroup with respect to a suitable discrete sampling measure in time, instead of fixing a uniform mixing time. However, as the works of Meyn and Tweedie show, petiteness and smallness are intimately related, and if a discrete-time chain is irreducible and aperiodic, then every petite set is indeed small (Theorem 5.5.7. of \cite{mtbook}). Despite this equivalence, as clearly shown in our setting, petiteness properties are much easier to verify, even in degenerate cases. This appears as one of Meyn and Tweedie's theory main points (see Commentary 5.6 of \cite{mtbook}, for example), and the implications are strong when petiteness can be checked for all compact sets. Similar strategies could turn out to be useful when trajectorial coupling conditions or ``\textit{mass-ratio control}" conditions as the ones discussed by~\cite{CloezGabriel20} prove hard to verify. Finally, it is worth noticing that a similar approach can be followed in a PDE framework by constructing ``\textit{controllabilty sets}", in the sense discussed in P.L. Lions' lectures \cite{pllions}, which play a role equivalent to petite sets in that theory. 

As commented in the Introduction, PDMP evolving at higher dimensions could be of particular interest for sampling complex target distributions in recent MCMC methods, as in the models studied by \cite{bouncy,zigzag,fetique}. Concerning the extension of the minorisation condition to higher dimensions, this does not seems to be much of an issue. The bound relies on the Duhamel representation \eqref{eq:Duhamel} of the semigroup $P_t$ which would be identical in a higher dimensional case. Therefore, if $\gamma_t(s,\mathbf{z}) := \varphi^{t-s}(0,\mathbf{z})$ is still a $C^1$-diffeomorphism on $\mathcal{X} \subset \mathbb{R}^{n+1}_+$, and suitable assumptions are made to bound uniformly from below the integrals, the extension of the presented method should be possible.

Finally, as pointed out by one of the reviewers, it could be also interesting to prove Large Deviations Asymptotics, by employing the theory presented by~\cite{Kontoyiannis2005LargeDA}, which extends the approach presented here to a \textit{multiplicative} ergodic theory. Such estimates are also interesting from biological points of view, where the Large Deviations rate function can be used to obtain variational representations of the Malthusian parameter, as shown for example in~\cite{baake}. This could be the subject of future works.
}

\bmhead{Acknowledgments}
	I thank Sylvie M\'el\'eard for her continual guidance and her many readings and suggestions for this manuscript, and Meriem El Karoui who introduced me to the biological motivations behind this model and has also guided me throughout the work. I would also like to thank Bertrand Cloez and Marie Doumic for the rewarding discussions during this work. Finally, I thank the anonymous reviewers for their interesting and insightful comments.
	
\bmhead{Funding Note} 
This work has been supported by the Chair “Mod\'elisation Math\'ematique et Biodiversit\'e” of Veolia Environnement - \'Ecole polytechnique - Museum National d'Histoire Naturelle - Fondation X. Funded by the European Union (ERC, SINGER, 101054787). Views and opinions expressed are however those of the author only and do not necessarily reflect those of the European Union or the European Research Council. Neither the European Union nor the granting authority can be held responsible for them.



\end{document}